\documentclass[a4paper]{amsart}

\usepackage{amssymb,enumitem,textcomp}
\usepackage[all]{xy}
\usepackage{hyperref,aliascnt}
\usepackage{mathtools,array}
\usepackage{centernot}
\usepackage{stmaryrd}

\numberwithin{equation}{section}

\newtheorem{lma}{Lemma}[section]

\newaliascnt{thmCt}{lma}
\newtheorem{thm}[thmCt]{Theorem}
\aliascntresetthe{thmCt}

\newaliascnt{corCt}{lma}

\aliascntresetthe{corCt}

\newaliascnt{prpCt}{lma}
\newtheorem{prp}[prpCt]{Proposition}
\aliascntresetthe{prpCt}

\newtheorem*{thm*}{Theorem}
\newtheorem*{cor*}{Corollary}
\newtheorem*{prp*}{Proposition}

\theoremstyle{definition}

\newaliascnt{pgrCt}{lma}
\newtheorem{pgr}[pgrCt]{}
\aliascntresetthe{pgrCt}

\newaliascnt{dfnCt}{lma}
\newtheorem{dfn}[dfnCt]{Definition}
\aliascntresetthe{dfnCt}

\newaliascnt{rmkCt}{lma}
\newtheorem{rmk}[rmkCt]{Remark}
\aliascntresetthe{rmkCt}

\newaliascnt{rmksCt}{lma}
\newtheorem{rmks}[rmksCt]{Remarks}
\aliascntresetthe{rmksCt}

\newaliascnt{qstCt}{lma}

\aliascntresetthe{qstCt}

\newaliascnt{pbmCt}{lma}
\newtheorem{pbm}[pbmCt]{Problem}
\aliascntresetthe{pbmCt}

\newaliascnt{exaCt}{lma}
\newtheorem{exa}[exaCt]{Example}
\aliascntresetthe{exaCt}

\newaliascnt{exasCt}{lma}

\aliascntresetthe{exasCt}

\newaliascnt{conjCt}{lma}

\aliascntresetthe{conjCt}

\newaliascnt{ntnCt}{lma}
\newtheorem{ntn}[ntnCt]{Notation}
\aliascntresetthe{ntnCt}

\newcommand{\NN}{\mathbb{N}}
\newcommand{\QQ}{\mathbb{Q}}
\newcommand{\RR}{\mathbb{R}}

\newcommand{\ZZ}{\mathbb{Z}}
\newcommand{\KK}{\mathcal{K}}
\newcommand{\pom}{positively ordered monoid}
\newcommand{\ca}{$C^*$-al\-ge\-bra}

\newcommand{\starHom}{${}^\ast$-ho\-mo\-mor\-phism}
\newcommand{\axiomO}[1]{(O#1)}
\newcommand{\freeVar}{\_\,}

\newcommand{\ihom}[1]{\llbracket #1 \rrbracket}
\newcommand{\cpc}[1]{\mathrm{cpc}_\perp\!( #1)}

\newcommand{\pathCu}[1]{\mathbf{#1}}

\newcommand{\txtSoft}{\text{soft}}
\newcommand{\nll}{\centernot\ll}
\newcommand{\stHom}{${}^*$-ho\-mo\-mor\-phism}
\newcommand{\CuSgp}{$\CatCu$-sem\-i\-group}

\newcommand{\CuMor}{$\CatCu$-mor\-phism}
\newcommand{\NNbar}{\overline{\mathbb{N}}}
\newcommand{\PPbar}{\overline{\mathbb{P}}}

\DeclareMathOperator{\Cu}{Cu}

\DeclareMathOperator{\id}{id}
\DeclareMathOperator{\Hom}{Hom}

\DeclareMathOperator{\counit}{e}

\newcommand{\CatPom}{\mathrm{PoM}}
\newcommand{\CatPomMor}{\CatPom}
\newcommand{\CatPomBimor}{\mathrm{Bi}\CatPom}
\newcommand{\CatCa}{C^*}
\newcommand{\CatCu}{\ensuremath{\mathrm{Cu}}}
\newcommand{\CatCuMor}{\CatCu}
\newcommand{\CatCuBimor}{\mathrm{Bi}\CatCu}
\newcommand{\CatP}{\mathcal{P}}
\newcommand{\CatPMor}{\CatP}
\newcommand{\CatQ}{\mathcal{Q}}
\newcommand{\CatQMor}{\CatQ}
\newcommand{\CatQBimor}{\mathrm{Bi}\CatQ}

\DeclareMathOperator{\Ext}{Ext}

\begin{document}

\title{Abstract bivariant Cuntz semigroups}
\author{Ramon Antoine}
\author{Francesc Perera}
\author{Hannes Thiel}

\date{\today}

\address{
Ramon~Antoine and Francesc~Perera,
Departament de Matem\`{a}tiques,
Universitat Aut\`{o}noma de Barcelona,
08193 Bellaterra, Barcelona, Spain}
\email[]{ramon@mat.uab.cat; perera@mat.uab.cat}

\address{
Hannes~Thiel,
Mathematisches Institut,
Universit\"at M\"unster,
Einsteinstrasse 62, 48149 M\"unster, Germany}
\email[]{hannes.thiel@uni-muenster.de}



\begin{abstract}
We show that abstract Cuntz semigroups form a closed symmetric monoidal category.
Thus, given Cuntz semigroups $S$ and $T$, there is another Cuntz semigroup $\ihom{S,T}$ playing the role of morphisms from $S$ to $T$.
Applied to \ca{s} $A$ and $B$, the semigroup $\ihom{\Cu(A),\Cu(B)}$ should be considered as the target in analogues of the UCT for bivariant theories of Cuntz semigroups.

Abstract bivariant Cuntz semigroups are computable in a number of interesting cases.
We also show that order-zero maps between \ca{s} naturally define elements in the respective bivariant Cuntz semigroup.
\end{abstract}

\maketitle

\section{Introduction}
\label{sec:intro}

The Cuntz semigroup $\Cu(A)$ of a \ca{} $A$ is an invariant that plays an important role in the structure theory of \ca{s} and the related Elliott classification program. It is defined analogously to the Murray-von Neumann semigroup, $V(A)$, by using equivalence classes of positive elements instead of projections;
see \cite{Cun78DimFct}.
In general, however, the semigroup $\Cu(A)$ contains much more information than $V(A)$, and it is therefore also more difficult to compute.

The Cuntz semigroup has been successfully used in classification results, both in the simple and nonsimple setting.
For example, Toms constructed two simple AH-algebras that have the same Elliott invariant, but which are not isomorphic, a fact that is captured by their Cuntz semigroups;
see \cite{Tom08ClassificationNuclear}.
On the other hand, Robert classified (not necessarily simple) inductive limits of one-dimensional NCCW-complexes with trivial $K_1$-groups using the Cuntz semigroup; see \cite{Rob12LimitsNCCW}.

The connection of $\Cu(A)$, for a \ca{} $A$, with the Elliott invariant of $A$ has been explored in a number of instances;
see for example \cite{PerTom07Recasting}, \cite{BroPerTom08CuElliottConj} and \cite{Tik11CuFunctions}.
In fact, for the class of simple, unital, nuclear \ca{s} that are $\mathcal{Z}$-stable (that is, that tensorially absorb the Jiang-Su algebra $\mathcal{Z}$), the Elliott invariant and the Cuntz semigroup together with $K_1$ determine one another functorially;
see \cite{AntDadPerSan14RecoverElliott}.
When dropping the assumption of $\mathcal{Z}$-stability, it is not known whether the pair consisting of the Elliott invariant and the Cuntz semigroup provides a complete invariant for classification of simple, unital, nuclear \ca{s}.

It is therefore very interesting to study the structural properties of $\Cu(A)$, for a \ca{} $A$.
This study was initiated by Coward, Elliott and Ivanescu in \cite{CowEllIva08CuInv}, who introduced a category $\CatCu$ and showed that the assignment $A\mapsto\Cu(A)$ is a sequentially continuous functor from \ca{s} to $\CatCu$.
The objects of $\CatCu$ are called \emph{abstract Cuntz semigroups} or \emph{$\CatCu$-semigroups}.
Working in this category allows one to provide elegant algebraic proofs for structural properties of \ca{s}.

A systematic study of the category $\CatCu$ was undertaken in \cite{AntPerThi14arX:TensorProdCu}.
One of the main results obtained is that $\CatCu$ is naturally a symmetric monoidal category (see \autoref{pgr:prelim:closedCat} for more details).
This means, in particular, that $\CatCu$ admits tensor products and that there is a bifunctor
\[
\otimes\colon\CatCu\times\CatCu\to\CatCu
\]
which is (up to natural isomorphisms) associative, symmetric, and has a unit object, namely the semigroup $\NNbar=\{0,1,2,\dots,\infty\}$.
The basic properties of this construction were studied in \cite{AntPerThi14arX:TensorProdCu}, relating in particular $\Cu(A\otimes B)$ with $\Cu(A)\otimes\Cu(B)$ for certain classes of \ca{s}.

An important motivation for our investigations here is to find an analogue of the universal coefficient theorem (UCT) for Cuntz semigroups.
Recall that a separable \ca{} $A$ is said to satisfy the UCT if for every separable \ca{} $B$ there is a short exact sequence
\[
0 \to \bigoplus_{i=0,1} \Ext\big( K_i(A), K_{1-i}(B) \big)
\to KK_0(A,B)
\to \bigoplus_{i=0,1} \Hom\big( K_i(A), K_i(B) \big)
\to 0.
\]
We refer to \cite[Chapter~23]{Bla98KThy} for details.

The goal is then to replace $KK_0(A,B)$ by a suitable bivariant version of the Cuntz semigroup (for example, along the lines of \cite{BosTorZac16arX:BivarThyCu}), and the $\Hom$-functor in the category of abelian groups by a suitable internal-hom functor in the category $\CatCu$.
In this direction, the construction developed in \cite{BosTorZac16arX:BivarThyCu} as a possible substitute for $KK_0(A,B)$ uses certain equivalence classes of completely positive contractive (abbreviated c.p.c.) order-zero maps between \ca{s}, denoted here as $\cpc{A,B}$.

The substitute of the $\Hom$-functor in the exact sequence above should be, if it exists, the adjoint to the tensor product functor alluded to above. It is thus a very natural question to determine whether the  category $\CatCu$ is, besides symmetric monoidal, also \emph{closed}.
This problem was left open in  \cite[Chapter~9]{AntPerThi14arX:TensorProdCu}.
More precisely, given \CuSgp{s} $S$ and $T$, the question asks if there exists a \CuSgp{} $\ihom{S,T}$ that plays the role of morphisms from $S$ to $T$, and such that the functor $\ihom{T,\freeVar}$ is adjoint to the functor $\freeVar\otimes T$. This means that, for any other \CuSgp{} $P$, we have a natural bijection
\[
\CatCuMor\big( S, \ihom{T,P} \big) \cong \CatCuMor\big( S\otimes T, P \big),
\]
where $\CatCuMor(\freeVar,\freeVar)$ denotes the set of morphisms in the category $\CatCu$.
The morphisms in $\CatCu$, also called \emph{\CuMor{s}}, are order-preserving monoid maps that preserve suprema of increasing sequences and that preserve the so-called  \emph{way-below relation};
see \autoref{dfn:prelim:CatCu}.
By functoriality, the natural model for $\CatCu$-morphisms consists of the \stHom{} between \ca{s}.

One of the main objectives of this paper is to construct the \CuSgp{} $\ihom{S,T}$ and to study its basic properties.
We call $\ihom{S,T}$ an \emph{abstract bivariant Cuntz semigroup} or a \emph{bivariant \CuSgp{}}.
The construction defines a bifunctor
\[
\ihom{\freeVar,\freeVar}\colon\CatCu\times\CatCu\to\CatCu,
\]
referred to as the \emph{internal-hom} bifunctor;
see, for example, \cite{Kel05EnrichedCat}.

The said construction resorts to the use of a more general class of maps than just \CuMor{s}.
A \emph{generalized \CuMor{}} is defined as an order-preserving monoid map that preserves suprema of increasing sequences (but not necessarily the way-below relation);
see \autoref{dfn:prelim:CatCu}. The natural model for generalized $\CatCu$-morphisms comes from order-zero maps between \ca{s}. We denote the set of such maps by $\CatCuMor[S,T]$.
Since every \CuMor{} is also a generalized \CuMor, we have an inclusion $\CatCuMor(S,T)\subseteq\CatCuMor[S,T]$.

When equipped with pointwise order and addition, $\CatCuMor[S,T]$ has a natural structure as a partially ordered monoid, but it is in general not a \CuSgp.
Similarly, $\CatCu(S,T)$ is usually not a \CuSgp{}.
The solution is to consider \emph{paths} in $\CatCuMor[S,T]$, that is, rationally indexed maps $\QQ\cap(0,1)\to\CatCu[S,T]$ that are `rapidly increasing' in a certain sense.
Equipped with a suitable equivalence relation, these paths define the desired \CuSgp{} $\ihom{S,T}$.

This procedure can be carried out in a much more general setting.
In \autoref{sec:Q} we introduce a category $\CatQ$ of partially ordered semigroups that, roughly speaking, is a weakening of the category $\CatCu$, in that the way-below relation is replaced by a possibly different binary relation (called \emph{auxiliary relation}).
We show that $\CatCu$ is a full subcategory of $\CatQ$;
see \autoref{prp:Q:fullCuQ}.
The path construction we have delineated above yields a covariant functor
\[
\tau\colon\CatQ\to\CatCu,
\]
that turns out to be right adjoint to the natural inclusion functor from $\CatCu$ into $\CatQ$;
see \autoref{thm:Q:coreflection}.
We refer to this functor as the \emph{$\tau$-construction}.
In \cite{AntPerThi17pre:Productscoproducs} we show that the $\tau$-construction can also be used to compute the Cuntz semigroup of ultraproduct \ca{s}.
In our setting, the functor $\tau$ applied to the semigroup of generalized \CuMor{s} $\CatCu[S,T]$ yields the internal-hom of $S$ and $T$;
see \autoref{dfn:bivarCu:ihom}.
In other words, for \CuSgp{s} $S$ and $T$, we define
\[
\ihom{S,T} := \tau(\CatCu[S,T]).
\]

We illustrate our results by computing a number of examples, that include the (Cuntz semigroups of the) Jiang-Su algebra $\mathcal{Z}$, the Jacelon-Razak algebra $\mathcal W$, UHF-algebras of infinite type, and purely infinite simple \ca{s}.
Interestingly, $\ihom{\Cu(\mathcal W),\Cu(\mathcal{W})}$ is isomorphic to the Cuntz semigroup of a II$_1$-factor.

The fact that $\CatCu$ is a closed category automatically adds additional features well known to category theory.
For example, one obtains a \emph{composition product} given in the form of a $\CatCu$-morphism:
\[
\circ\colon\ihom{T,P}\otimes\ihom{S,T}\to\ihom{S,P}.
\]
In particular, this product equips $\ihom{S,S}$ with the structure of a (not necessarily commutative) $\CatCu$-semiring.
These structures will be further analysed in a subsequent paper; see \cite{AntPerThi17pre:AbsbivarII}.

Finally, we specialise to \ca{s} and show that a c.p.c.\ order-zero map $\varphi\colon A\to B$ between \ca{s} $A$ and $B$ naturally defines an element $\Cu(\varphi)$ in the bivariant \CuSgp{} $\ihom{\Cu(A),\Cu(B)}$;
see \autoref{prp:appl:oz:oz_ind_ihom} and \autoref{dfn:appl:oz:oz_ind_ihom}.
We then analyse the induced map
\[
\cpc{A,B}\to\ihom{\Cu(A),\Cu(B)},
\]
and show it is surjective in a number of cases; namely for a UHF-algebra of infinite type, the Jiang-Su algebra, or the Jacelon-Razak algebra $\mathcal{W}$; see \autoref{exa:appl:cpcToIhom}.

\section*{Acknowledgements}

This work was initiated during a research in pairs (RiP) stay at the Oberwolfach Research Institute for Mathematics (MFO) in March 2015.
The authors would like to thank the MFO for financial support and for providing inspiring working conditions.

Part of this research was conducted while the third named author was visiting the Universitat Aut\`{o}noma de Barcelona (UAB) in September 2015 and June 2016, and while the first and second named authors visited M\"unster Universit\"at in June 2015 and 2016. Part of the work was also completed while the second and third named authors were attending the Mittag-Leffler institute during the 2016 program on Classification of Operator Algebras: Complexity, Rigidity, and Dynamics.
They would like to thank all the involved institutions for their kind hospitality.

The two first named authors were partially supported by DGI MICIIN (grant No.\ MTM2011-28992-C02-01), by MINECO (grant No.\ MTM2014-53644-P), and by the Comissionat per Universitats i Recerca de la Generalitat de Catalunya.
The third named author was partially supported by the Deutsche Forschungsgemeinschaft (SFB 878).

\section{Preliminaries}
\label{sec:prelim}

Throughout, $\KK$ denotes the \ca{} of compact operators on a separable, infinite-dimensional Hilbert space.
Given a \ca{} $A$, we let $A_+$ denote the positive elements in $A$.

\subsection{The category \texorpdfstring{$\CatCu$}{Cu} of abstract Cuntz semigroups}
\label{sec:prelim:CatCu}

In this subsection, we recall the definition of the category $\CatCu$ of abstract Cuntz semigroups.

\begin{pgr}
\label{pgr:prelim:CatPom}
We first recall the basic theory of the category $\CatPom$ of \pom{s};
see \cite[Appendix~B.2]{AntPerThi14arX:TensorProdCu} for details.
A \emph{\pom} is a commutative monoid $M$ together with a partial order $\leq$ such that $a\leq b$ implies that $a+c\leq b+c$ for all $a,b,c\in M$, and such that $0\leq a$ for all $a\in M$.
We let $\CatPom$ denote the category whose objects are \pom{s}, and whose morphisms are maps preserving addition, order and the zero element.

Let $M,N$ and $P$ be \pom{s}.
We denote the set of $\CatPom$-morphisms from $M$ to $N$ by $\CatPomMor(M,N)$.
A map $\varphi\colon M\times N\to P$ is called a \emph{$\CatPom$-bimorphism} if it is a $\CatPom$-morphism in each variable.
We denote the collection of such maps by $\CatPomBimor(M\times N,P)$.
We equip both $\CatPomMor(M,N)$ and $\CatPomBimor(M\times N,P)$ with pointwise order and addition, which gives them a natural structure as \pom{s}.

Given \pom{s} $M$ and $N$, there exists a \pom{} $M\otimes_{\CatPom}N$ and a $\CatPom$-bimorphism $\omega\colon M\times N\to M\otimes_{\CatPom}N$ with the following universal property:
For every $P$, mapping a $\CatPom$-morphism $\alpha\colon M\otimes_{\CatPom}N\to P$ to the $\CatPom$-bimorphism $\alpha\circ\omega\colon M\times N\to P$ defines a natural bijection
\[
\CatPomMor \big( M\otimes_{\CatPom}N, P \big) \cong \CatPomBimor \big( M\times N, P \big),
\]
which moreover respects the structure of the (bi)morphism sets as \pom{s}.
We call $M\otimes_{\CatPom}N$ together with $\omega$ the \emph{tensor product} of $M$ and $N$. 
\end{pgr}

Recall that a set $\Lambda$ with a binary relation $\prec$ is called \emph{upward directed} if for all $\lambda_1,\lambda_2\in\Lambda$ there exists $\lambda\in\Lambda$ with $\lambda_1,\lambda_2\prec\lambda$.
Following \cite[Definition~I-1.11, p.57]{GieHof+03Domains}, we define auxiliary relations on partially ordered sets and monoids:

\begin{dfn}
\label{dfn:prelim:auxRel}
Let $X$ be a partially ordered set.
An \emph{auxiliary relation} on $X$ is a binary relation $\prec$ on $X$ satisfying the following conditions for all $x,x',y,y'\in X$:
\begin{enumerate}
\item
If $x\prec y$ then $x\leq y$.
\item
If $x'\leq x\prec y\leq y'$ then $x'\prec y'$.
\end{enumerate}
If $X$ is also a monoid, then an auxiliary relation $\prec$ on $X$ is said to be \emph{additive} if $0\prec x$ for all $x\in X$ and if for all $x_1,x_2,y_1,y_2\in X$ with $x_1\prec y_1$ and $x_2\prec y_2$ we have $x_1+x_2\prec y_1+y_2$.
\end{dfn}

An important example of an auxiliary relation is the so called \emph{way-below relation}, which has its origins in domain theory (see \cite{GieHof+03Domains}).
We recall below its sequential version, which is the one used to define abstract Cuntz semigroups.

\begin{dfn}
\label{dfn:prelim:waybelow}
Let $X$ be a partially ordered set, and let $x,y\in X$.
We say that $x$ is \emph{way-below} $y$, or that $x$ is \emph{compactly contained in} $y$, in symbols $x\ll y$, if whenever $(z_n)_n$ is an increasing sequence in $X$ for which the supremum exists and which satisfies $y\leq \sup_n z_n$, then there exists $k\in\NN$ with $x\leq z_k$.
We say that $x$ is \emph{compact} if $x\ll x$.
We let $X_c$ denote the set of compact elements in $X$.
\end{dfn}

The following definition is due to Coward, Elliott and Ivanescu in \cite{CowEllIva08CuInv}.
See also \cite[Definition~3.1.2]{AntPerThi14arX:TensorProdCu}.

\begin{dfn}
\label{dfn:prelim:CatCu}
A \emph{$\CatCu$-semigroup}, also called \emph{abstract Cuntz semigroup}, is a positively ordered semigroup $S$ that satisfies the following axioms \axiomO{1}-\axiomO{4}:
\begin{itemize}
\item[\axiomO{1}]
Every increasing sequence $(a_n)_n$ in $S$ has a supremum $\sup_n a_n$ in $S$.
\item[\axiomO{2}]
For every element $a\in S$ there exists a sequence $(a_n)_n$ in $S$ with $a_n\ll a_{n+1}$ for all $n\in\NN$, and such that $a=\sup_n a_n$.
\item[\axiomO{3}]
If $a'\ll a$ and $b'\ll b$ for $a',b',a,b\in S$, then $a'+b'\ll a+b$.
\item[\axiomO{4}]
If $(a_n)_n$ and $(b_n)_n$ are increasing sequences in $S$, then $\sup_n(a_n+b_n)=\sup_n a_n+\sup_n b_n$.
\end{itemize}

Given $\CatCu$-semigroups $S$ and $T$, a \emph{$\CatCu$-morphism} is a map $f\colon S\to T$ that preserves addition, order, the zero element, the way-below relation and suprema of increasing sequences.
A \emph{generalized $\CatCu$-morphism} is a $\CatCu$-morphism that is not required to preserve the way-below relation.
We denote the set of $\CatCu$-morphisms by $\CatCuMor(S,T)$;
and we denote the set of generalized $\CatCu$-morphisms by $\CatCuMor[S,T]$.

We let $\CatCu$ be the category whose objects are $\CatCu$-semigroups and whose morphisms are $\CatCu$-morphisms.
\end{dfn}

\begin{rmk}
\label{rmk:prelim:CatCu}
Let $S$ be a $\CatCu$-semigroup.
Note that $0\ll a$ for all $a\in S$.
Thus, \axiomO{3} ensures that $\ll$ is an additive auxiliary relation on $S$.
\end{rmk}

\begin{pgr}
Let $A$ be a \ca{}, and let $a,b\in(A\otimes\KK)_+$.
We say that $a$ is \emph{Cuntz subequivalent} to $b$, denoted $a\precsim b$, if there is a sequence $(x_n)_n$ in $A\otimes \KK$ such that $a=\lim_n x_nbx_n^*$.
We say that $a$ and $b$ are \emph{Cuntz equivalent}, written $a\sim b$, provided $a\precsim b$ and $b\precsim a$.
The set of equivalence classes $\Cu(A) :=(A\otimes\KK)_+/\!\!\sim$ is called the (completed) Cuntz semigroup of $A$.
One defines an addition on $\Cu(A)$ by setting $[a]+[b]:=\left[\left(\begin{smallmatrix}a & 0\\ 0 & b\end{smallmatrix}\right)\right]$ for $a,b\in(A\otimes\KK)_+$.
(One uses that there is an isomorphism $M_2(\KK)\cong\KK$, and that the definition does not depend on the choice of isomorphism.)
The class of $0\in (A\otimes\KK)_+$ is a zero element for $\Cu(A)$. 
One defines an order on $\Cu(A)$ by setting $[a]\leq [b]$ whenever $a\precsim b$.
This gives $\Cu(A)$ the structure of a \pom.
\end{pgr}

\begin{thm}[\cite{CowEllIva08CuInv}]
\label{prp:prelim:functorCu}
For every \ca{} $A$, the \pom{} $\Cu(A)$ is a $\CatCu$-semigroup.
Furthermore, if $B$ is another \ca{}, then a \stHom{} $\varphi\colon A\to B$ induces a $\CatCu$-morphism $\Cu(\varphi)\colon\Cu(A)\to\Cu(B)$ by
\[
\Cu(\varphi)([a]) := [\varphi(a)],
\]
for $a\in(A\otimes\KK)_+$.
This defines a functor from the category of \ca{s} with \starHom{s} to the category $\CatCu$.
\end{thm}

\begin{rmk}
Let $A$ be a \ca{}.
In order to show that \axiomO{2} holds for $\Cu(A)$ one proves that, for every $a\in (A\otimes K)_+$ and $\varepsilon>0$ we have $[(a-\varepsilon)_+]\ll [a]$, and that moreover $[a]=\sup_{\varepsilon>0} [(a-\varepsilon)_+]$.
One can then derive from this that the sequence $([(a-1/n)_+])_n$ satisfies the required properties in \axiomO{2}.

This suggests the possibility of formally strengthening \axiomO{2} for every \CuSgp{} $S$ in the following way:
Given $a\in S$, there exists a $(0,1)$-indexed chain of elements $(a_\lambda)_{\lambda\in (0,1)}$ with the property that $a=\sup_{\lambda}a_\lambda$, and $a_{\lambda'}\ll a_{\lambda}$ whenever $\lambda'<\lambda$.
Next, we show that this property holds for all \CuSgp{s}.
\end{rmk}

\begin{lma}
\label{prp:prelim:ctsO2starCondition}
Let $S$ be a set equipped with a transitive binary relation $\prec$ that satisfies the following condition:
\begin{enumerate}
\item[(*)]
For each $a\in S$ there exists a sequence $(a_n)_n$ in $S$ such that $a_n\prec a_{n+1}\prec a$ for all $n$;
and such that whenever $a'\in S$ satisfies $a'\prec a$ then there exists $n_0$ with $a'\prec a_{n_0}$.
\end{enumerate}
Then, for every $a\in S$, there exists a chain $(a_{\lambda})_{\lambda\in (0,1)\cap \QQ}$ such that $a_{\lambda'}\prec a_\lambda$ whenever $\lambda',\lambda\in(0,1)\cap\QQ$ satisfy $\lambda'<\lambda$, and such that for every $a'\in S$ with $a'\prec a$ there exists $\mu\in (0,1)\cap\QQ$ with $a'\prec a_\mu$.
\end{lma}
\begin{proof}
Note that condition (*) implies the following:
Whenever $b_1,b_2,b\in S$ satisfy $b_1,b_2\prec b$, then there exists $b_3\in S$ with $b_1,b_2\prec b_3\prec b$.
This property, which we will refer to as the interpolation property, will be used throughout.

Given $a\in S$, first use (*) to fix an increasing sequence $0\prec a_1\prec a_2\prec \dots \prec a$ which is cofinal in $a^\prec:=\{b \mid b\prec a\}$.
(This means that, if $a'\in S$ satisfies $a'\prec a$, then there is $k\in\NN$ with $a'\prec a_k$.)
Use the interpolation property to find $a_1^{(1)}$ such that $a_1\prec a_1^{(1)}\prec a$ and consider the chain $0\prec a_1^{(1)}\prec a$.
Now use the interpolation property to refine the above chain as
\[
\begin{array}{rcccl}
0 & \prec & a_1^{(1)} & \prec & a \\
\rotatebox[origin=c]{90}{=} & & \rotatebox[origin=c]{90}{=} & & \rotatebox[origin=c]{90}{=} \\ 0 &\prec a_1^{(2)} \prec & a_2^{(2)} & \prec a_3^{(2)}\prec & a \,,
\end{array}\]
in such a way that moreover $a_2\prec a_3^{(2)}$.
We now proceed inductively, and thus suppose we have constructed a chain $0\prec a_1^{(n)}\prec \dots\prec a_{2^n-1}^{(n)}\prec a$ with $a_n\prec a_{2^n-1}^{(n)}$.
Use the interpolation property to construct a new chain
\[
0\prec a_1^{(n+1)}\prec \dots \prec a_{2^{n+1}-1}^{(n+1)}\prec a
\]
such that
\[
0\prec a_1^{(n+1)}\prec a_1^{(n)},\quad
a_{i}^{(n)}\prec a_{2i+1}^{(n+1)}\prec a_{i+1}^{(n)},\quad
a_{2i}^{(n+1)}=a_{i}^{(n)},\quad
a_{2^n-1}^{(n)}\prec a_{2^{n+1}-1}^{(n+1)}\prec a,
\]
and such that moreover $a_{n+1}\prec a_{2^{n+1}-1}^{(n+1)}$.
This latter condition will ensure that the set of elements thus constructed is cofinal in $a^\prec$.

The index set $I:=\{(n,i)\mid 1\leq n, 1\leq i\leq 2^n-1\}$ can be totally ordered by setting $(n,i)\leq (m,j)$ provided $i2^{-n}\leq j2^{-m}$.
It now follows from the construction above that $a^{(n)}_i\prec a^{(m)}_j$ whenever $(n,i)\leq (m,j)$.

The set $I$ is order-isomorphic to the dyadic rationals in $(0,1)$.
In fact, $I$ is a countably infinite, totally ordered, dense set with no minimal nor maximal element.
(Here, dense means that whenever $x<y$ in $I$ there exists $z\in I$ with $x<z<y$.)
By a classical result of G. Cantor (see, for example, \cite[Theorem~27]{Roi90}), there is only one such set, up to order-isomorphism.
We can therefore choose an order-isomorphism $\psi\colon I\to (0,1)\cap \QQ$ and set $a_\lambda=a^{(n)}_i$ whenever $\psi((n,i))=\lambda$.
\end{proof}

The proof of the following proposition is (essentially) included in \cite[Proposition IV-3.1]{GieHof+03Domains}.

\begin{prp}
\label{prp:prelim:ctsO2}
Let $S$ be a \CuSgp, and let $a\in S$.
Then, there exists a family $(a_\lambda)_{\lambda\in(0,1]}$ in $S$ with $a_1=a$;
such that $a_{\lambda'}\ll a_\lambda$ whenever $\lambda',\lambda\in(0,1]$ satisfy $\lambda'<\lambda$;
and such that $a_\lambda=\sup_{\lambda'<\lambda} a_{\lambda'}$ for every $\lambda\in(0,1]$.
\end{prp}
\begin{proof}
Consider $S$ equipped with the transitive relation $\ll$.
Then \axiomO{2} ensures that condition (*) in \autoref{prp:prelim:ctsO2starCondition} is fulfilled with $\ll$ in place of $\prec$.
Hence, given $a\in S$ we can apply \autoref{prp:prelim:ctsO2starCondition} to choose a $\ll$-increasing chain $(\bar a_\lambda)_{\lambda\in (0,1)\cap\QQ}$ with $a=\sup_\lambda \bar a_\lambda$.
For each $\lambda\in(0,1]$, define $a_\lambda:=\sup\{\bar a_{\lambda'} : \lambda'<\lambda\}$.
It is now easy to see that the chain $(a_\lambda)_{\lambda\in (0,1]}$ satisfies the conclusion.
\end{proof}

\subsection{Closed, monoidal categories}
\label{sec:prelim:cats}

In this subsection, we recall the basic notions from the theory of closed, monoidal categories.
For details we refer to \cite{Kel05EnrichedCat} and \cite{MacLan71Categories}.
See also \cite[Appendix~A]{AntPerThi14arX:TensorProdCu}.

\begin{pgr}
\label{pgr:prelim:monCat}
A monoidal category $\mathcal{V}$ consists of a category $\mathcal{V}_0$ (which we assume is locally small), a bifunctor $\otimes\colon \mathcal{V}_0\times \mathcal{V}_0\to \mathcal{V}_0$ (covariant in each variable) and a unit object $I$ in $\mathcal{V}_0$ such that, whenever $X,Y,Z$ are objects in $\mathcal{V}_0$, there are natural isomorphisms $(X\otimes Y)\otimes Z\cong X\otimes (Y\otimes Z)$, and $X\otimes I\cong X$, and $I\otimes X\cong X$,
that are subject to certain coherence axioms.
An object or morphism in $\mathcal{V}$ means an object or morphism in $\mathcal{V}_0$, respectively.
In concrete examples, such as $\CatPom$ and $\CatCu$, we will use the same notation for a monoidal category and its underlying category.

A monoidal category $\mathcal{V}$ is called \emph{symmetric} provided that for each pair of objects $X$ and $Y$ there is a natural isomorphism $X\otimes Y\cong Y\otimes X$.

In many concrete examples of monoidal categories, the tensor product of two objects $X$ and $Y$ is the object $X\otimes Y$ (unique up to natural isomorphism) that linearizes bilinear maps from $X\times Y$.
This is formalized by considering a functorial association of bimorphisms $\mathrm{Bimor}(X\times Y,Z)$ such that $X\otimes Y$ represents the functor $\mathrm{Bimor}(X\times Y,\freeVar)$, that is, for each $Z$ there is a natural bijection
\[
\mathrm{Bimor}\big( X\times Y,Z \big) \cong \mathrm{Mor}\big( X\otimes Y,Z \big).
\]

One instance of this is the monoidal structure in the category $\CatCu$ of abstract Cuntz semigroups. We recall details in \autoref{sec:prelim:tensCu}.
Another example is the category $\CatPom$ of \pom{s};
see \autoref{pgr:prelim:CatPom}.
\end{pgr}

\begin{pgr}
\label{pgr:prelim:closedCat}
A monoidal category $\mathcal V$ is said to be \emph{closed} provided that for each object $Y$, the functor $-\otimes Y\colon \mathcal V_0\to \mathcal V_0$ has a right adjoint, that we will denote by $\ihom{Y,-}$.
Thus, in a closed monoidal category, for all objects $X,Y,Z$, there is a natural bijection
\[
\mathcal{V}_0\big( X\otimes Y,Z \big)\cong \mathcal{V}_0\big( X,\ihom{Y,Z} \big),
\]
where $\mathcal{V}_0(\freeVar,\freeVar)$ denotes the morphisms between two objects $X$ and $Y$.

Let $\mathcal{V}$ be a monoidal category with unit object $I$.
An \emph{enriched} category $\mathcal{C}$ over $\mathcal{V}$ consists of:
a collection of objects in $\mathcal{C}$;
an object $\mathcal{C}(X,Y)$ in $\mathcal{V}$, for each pair of objects $X$ and $Y$ in $\mathcal{C}$ (playing the role of the morphisms in $\mathcal{C}$ from $X$ to $Y$);
a $\mathcal{V}$-morphism $j_X\colon I\to\mathcal{C}(X,X)$, called the identity on $X$, for each object $X$ in $\mathcal{C}$ (playing the role of the identity morphism on $X$);
and for each triple $X$, $Y$ and $Z$ of objects in $\mathcal{C}$, a $\mathcal{V}$-morphism $\mathcal{C}(Y,Z)\otimes\mathcal{C}(X,Y)\to\mathcal{C}(X,Z)$ that plays the role of a composition law and is subject to certain coherence axioms;
see \cite[Section 1.2]{Kel05EnrichedCat} for details.

It follows from general category theory that every closed symmetric monoidal category $\mathcal{V}$ can be enriched over itself.
Let us recall some details.
Given two objects $X$ and $Y$ in $\mathcal{V}$, the object $\ihom{X,Y}$ in $\mathcal{V}$ plays the role of the morphisms from $X$ to $Y$.
Given an object $X$, the identity on $X$ (for the enrichment) is defined as the $\mathcal{V}$-morphism $j_X\colon I\to\ihom{X,X}$ that corresponds to the `usual' identity morphism $\id_X\in\mathcal{V}_0\big( X, X \big)$ under the following natural bijections
\[
\mathcal{V}_0\big( I,\ihom{X,X} \big)
\cong \mathcal{V}_0\big( I\otimes X, X \big)
\cong \mathcal{V}_0\big( X, X \big).
\]
It is easiest to construct the composition map by using the evaluation maps.
Given objects $X$ and $Y$, the evaluation (or counit) map is defined as the $\mathcal{V}$-morphism $\counit_X^Y\colon\ihom{X,Y}\otimes X\to Y$ that corresponds to the identity morphism in $\mathcal{V}_0(\ihom{X,Y},\ihom{X,Y})$ under the natural bijection
\[
\mathcal{V}_0\big( \ihom{X,Y}\otimes X, Y \big)
\cong \mathcal{V}_0\big( \ihom{X,Y},\ihom{X,Y} \big).
\]
Then, given objects $X$, $Y$ and $Z$, the composition $\ihom{Y,Z}\otimes\ihom{X,Y}\to\ihom{X,Z}$ is defined as the $\mathcal{V}$-morphism that corresponds to the composition
\[
\ihom{Y,Z}\otimes\ihom{X,Y}\otimes X \xrightarrow{\id_{\ihom{Y,Z}}\otimes\counit_X^Y} \ihom{Y,Z}\otimes Z \xrightarrow{\counit_Y^Z} Z
\]
under the natural bijection
\[
\mathcal{V}_0\big( \ihom{Y,Z}\otimes\ihom{X,Y}, \ihom{X,Z} \big)
\cong \mathcal{V}_0\big( \ihom{Y,Z}\otimes\ihom{X,Y}\otimes X, Z \big).
\]
\end{pgr}

The natural question of whether $\CatCu$ is a closed category was left open in \cite[Problem~2]{AntPerThi14arX:TensorProdCu}.
We show in \autoref{prp:bivarCu:CuClosed} that this is indeed the case.

\subsection{Tensor products in \texorpdfstring{$\CatCu$}{Cu}}
\label{sec:prelim:tensCu}

In this subsection we recall the construction of tensor products of $\CatCu$-semigroups as introduced in \cite{AntPerThi14arX:TensorProdCu}.

\begin{dfn}[{\cite[Definition~6.3.1]{AntPerThi14arX:TensorProdCu}}]
\label{dfn:prelim:CatCuBimor}
Let $S,T$ and $P$ be $\CatCu$-semigroups, and let $\varphi\colon S\times T\to P$ be a $\CatPom$-bimorphism.
We say that $\varphi$ is a \emph{$\CatCu$-bimorphism} if it satisfies the following conditions:
\begin{enumerate}
\item
We have that $\sup_k\varphi(a_k,b_k)=\varphi(\sup_k a_k, \sup_k b_k)$, for every increasing sequences $(a_k)_k$ in $S$ and $(b_k)_k$ in $T$.
\item
If $a',a\in S$ and $b',b\in T$ satisfy $a'\ll a$ and $b'\ll b$, then $\varphi(a',b')\ll\varphi(a,b)$.
\end{enumerate}
We denote the set of $\CatCu$-bimorphisms by $\CatCuBimor(S\times T,R)$.
\end{dfn}

Given $\CatCu$-semigroups $S,T$ and $P$, we equip $\CatCuBimor(S\times T,R)$ with pointwise order and addition, giving it the structure of a \pom.
Similarly, we consider the set of $\CatCu$-morphisms between two $\CatCu$-semigroups as a \pom{} with the pointwise order and addition.

\begin{thm}[{\cite[Theorem~6.3.3]{AntPerThi14arX:TensorProdCu}}]
\label{prp:prelim:tensCu}
Let $S$ and $T$ be $\CatCu$-semigroups.
Then there exists a $\CatCu$-semigroup $S\otimes T$ and a $\CatCu$-bimorphism $\omega\colon S\times T\to S\otimes T$ such that for every $\CatCu$-semigroup $P$ the following universal properties hold:
\begin{enumerate}
\item
For every $\CatCu$-bimorphism $\varphi\colon S\times T\to P$ there exists a (unique) $\CatCu$-morphism $\tilde{\varphi}\colon S\otimes T\to P$ such that $\varphi=\tilde{\varphi}\circ\omega$.
\item
If $\alpha_1,\alpha_2\colon S\otimes T\to P$ are $\CatCu$-morphisms, then $\alpha_1\leq\alpha_2$ if and only if $\alpha_1\circ\omega\leq\alpha_2\circ\omega$.
\end{enumerate}
Thus, for every $P$, the assignment that sends a $\CatCu$-morphism $\alpha\colon S\otimes T\to P$ to the $\CatCu$-bimorphism $\alpha\circ\omega\colon S\times T\to P$ defines a natural bijection
\[
\CatCuMor\big( S\otimes T, P \big) \cong \CatCuBimor\big( S\times T, P \big),
\]
which respects the structure of the (bi)morphism sets as \pom{s}.
\end{thm}

\begin{pgr}
\label{pgr:prelim:tensmaps}
Let $S$ and $T$ be $\CatCu$-semigroups, and consider the universal $\CatCu$-bimorphism $\omega\colon S\times T\to S\otimes T$ from \autoref{prp:prelim:tensCu}.
Given $s\in S$ and $t\in T$, we set $s\otimes t :=\omega(s,t)$.
We call $s\otimes t$ a \emph{simple tensor}.

The tensor product in $\CatCu$ is functorial in each variable:
If $\varphi_1\colon S_1\to T_1$ and $\varphi_2\colon S_2\to T_2$ are $\CatCu$-morphisms, then there is a unique $\CatCu$-morphism $\varphi_1\otimes\varphi_2\colon S_1\otimes S_2\to T_1\otimes T_2$ with the property that $(\varphi_1\otimes\varphi_2)(a_1\otimes a_2)=\varphi_1(a_1)\otimes\varphi_2(a_2)$ for every $a_1\in S_1$ and $a_2\in S_2$.

Thus, the tensor product in $\CatCu$ defines a bifunctor $\otimes\colon\CatCu\times\CatCu\to\CatCu$.
The $\CatCu$-semigroup $\NNbar=\{0,1,2,\ldots,\infty\}$ is a unit object, that is, for every $\CatCu$-semigroup $S$ there are canonical isomorphisms $S\otimes\NNbar\cong S$ and $\NNbar\otimes S\cong S$.
Further, for every $\CatCu$-semigroups $S$, $T$ and $P$, there are natural isomorphisms
\[
S\otimes (T\otimes P) \cong (S\otimes T)\otimes P
\quad\text{ and }\quad
S\otimes T \cong T\otimes S.
\]
It follows that $\CatCu$ is a symmetric, monoidal category;
see also \cite[6.3.7]{AntPerThi14arX:TensorProdCu}.
\end{pgr}

\section{The Path Construction}
\label{sec:path}

In this section we introduce a functorial construction from a category of monoids with a transitive relation to the category $\CatCu$.
This construction, when restricted to the category $\CatQ$ introduced in \autoref{sec:Q} (a category that contains $\CatCu$) is a coreflection for the natural inclusion from $\CatCu$.

\begin{dfn}
\label{dfn:path:CatP}
A \emph{$\CatP$-semigroup} is a pair $(S,\prec)$, where $S$ is a commutative monoid and where $\prec$ is a transitive relation on $S$, such that:
\begin{enumerate}
\item
We have $0\prec a$ for all $a\in S$.
\item
If $a_1,a_2,b_1,b_2\in S$ satisfy $a_1\prec b_1$ and $a_2\prec b_2$, then $a_1+a_2\prec b_1+b_2$.
\end{enumerate}
We often denote a $\CatP$-semigroup $(S,\prec)$ simply by $S$.

A \emph{$\CatP$-morphism} is a monoid morphism that preserves the relation.
Given $\CatP$-semigroups $(S,\prec)$ and $(T,\prec)$, we denote the collection of all $\CatP$-morphisms by $\CatPMor((S,\prec),(T,\prec))$, or simply by $\CatPMor(S,T)$.
We let $\CatP$ be the category whose objects are $\CatP$-semigroups and whose morphisms are $\CatP$-morphisms.
\end{dfn}

\begin{rmk}
Conditions~(1) and~(2) of \autoref{dfn:path:CatP} are the same as the conditions from \autoref{dfn:prelim:auxRel} for an auxiliary relation to be additive.
\end{rmk}

\begin{dfn}
\label{dfn:path:path}
Let $I=(I,\prec)$ be a set with an upward directed transitive relation $\prec$.
Let $S=(S,\prec)$ be a $\CatP$-semigroup.
An \emph{$I$-path} (or simply a \emph{path}) in $S$ is a map $f\colon I\to S$ such that $f(\lambda')\prec f(\lambda)$ whenever $\lambda',\lambda\in I$ satisfy $\lambda'\prec\lambda$.
We set
\[
P(I,S) := \big\{ f\colon I\to S \text{ such that } f \text{ is a path in } S \big\}.
\]
We define the sum of two paths $f$ and $g$ setting $(f+g)(\lambda):=f(\lambda)+g(\lambda)$, for $\lambda\in I$.
Let $0\in P(I,S)$ denote the path given by $0(\lambda)=0$, for $\lambda\in I$.

We define a binary relation $\precsim$ on $P(I,S)$ by setting $f\precsim g$ for two paths $f$ and $g$ if and only if for every $\lambda\in I$ there exists $\mu\in I$ such that $f(\lambda)\prec g(\mu)$.
Finally we antisymmetrize the relation $\precsim$ by setting $f\sim g$ if and only if $f\precsim g$ and $g\precsim f$.

Given $s\in S$ and $f\in P(I,S)$, we write $s\prec f$ if $s\prec f(\lambda)$ for all $\lambda\in I$;
and we write $f\prec s$ provided $f(\lambda)\prec s$ for all $\lambda\in I$.
\end{dfn}

The proof of the following result is straightforward and therefore omitted.

\begin{lma}
\label{prp:path:relationsPaths}
Let $I$ be a set with an upward directed transitive relation, and let $S$ be a $\CatP$-semigroup.
Then the addition and the zero element defined in \autoref{dfn:path:path} give $P(I,S)$ the structure of a commutative monoid.
Moreover, the relation $\precsim$ on $P(I,S)$ is transitive, reflexive and satisfies:
\begin{enumerate}
\item
For every $f\in P(I,S)$ we have $0\precsim f$.
\item
If $f_1,f_2,g_1,g_2\in P(I,S)$ satisfy $f_1\precsim g_1$ and $f_2\precsim g_2$, then $f_1+f_2\precsim g_1+g_2$.
\end{enumerate}
Further, $\sim$ is an equivalence relation on $P(I,S)$.
\end{lma}

\begin{dfn}
\label{dfn:path:tau}
Let $I$ be a set with an upward directed transitive relation, and let $S$ be a $\CatP$-semigroup.
Let $\sim$ be the equivalence relation on $P(I,S)$ from \autoref{dfn:path:path}.
We define
\[
\tau_I(S) :=P(I,S)/_{\sim}.
\]
Given a path $f$ in $S$, its equivalence class in $\tau_I(S)$ is denoted by $[f]$.

We define $0\in \tau_I(S)$ as the equivalence class of the zero-path.
We define $+$ and $\leq$ on $\tau_I(S)$ by setting $[f]+[g]:=[f+g]$, and by setting $[f]\leq[g]$ provided $f\precsim g$.
\end{dfn}

The following results follows immediately from \autoref{prp:path:relationsPaths}.

\begin{prp}
\label{prp:path:tauispom}
Let $I$ be a set with an upward directed transitive relation, and let $S$ be a $\CatP$-semigroup.
Then the addition, the zero element, and the order defined in \autoref{dfn:path:tau} give $\tau_I(S)$ the structure of a \pom.
\end{prp}

\begin{rmks}
\label{rmk:path:tauispom}
(1)
We call the construction of $\tau_I(S)$ the \emph{$\tau$-construction} or \emph{path construction}.
We call $I$ the \emph{path type}.

(2)
Given a $\CatP$-semigroup $S$, the path construction $\tau_I(S)$ depends heavily on the choice of $I$.
For instance, using the most simple case $I=(\{0\},\leq)$, we obtain
\[
\tau_{\{0\}}(S)
\simeq \big\{ a\in S : a\prec a \big\}.
\]
For $I=(\NN,<)$, one can show that $\tau_I(S)$ is the (sequential) round ideal completion of $S$ as considered for instance in \cite[Proposition~3.1.6]{AntPerThi14arX:TensorProdCu}.

We will not pursue this general constructions further.
Rather,  motivated by the results in \autoref{prp:prelim:ctsO2starCondition} and \autoref{prp:prelim:ctsO2}, we will focus on the concrete case where the path type is taken to be $\big(\QQ\cap(0,1),<\big)$.
\end{rmks}

\begin{ntn}
\label{ntn:path:standardPathType}
We set $I_\QQ:=\big(\QQ\cap(0,1),<\big)$.
Given a $\CatP$-semigroup $S$, we denote $P(I_\QQ,S)$ and $\tau_{I_\QQ}(S)$ by $P(S)$ and $\tau(S)$, respectively.
If we want to stress the auxiliary relation on $S$, we also write $P(S,\prec)$ and $\tau(S,\prec)$.

Thinking of $I_\QQ$ as an ordered index set, we will often denote a path in $S$ as an indexed family $(a_\lambda)_{\lambda\in I_\QQ}$.
\end{ntn}

Given a $\CatP$-semigroup $S$, we show in \autoref{prp:path:pathinCu} that $\tau(S)$ is a $\CatCu$-semigroup when equipped with the order and addition in \autoref{dfn:path:tau}.
We split the proof into several lemmas.
Recall from \autoref{dfn:path:path} that, given paths $f$ and $g$ in $S$, and given $\lambda\in I_\QQ$, we write $f(\lambda)\prec g$ (respectively, $f\prec g(\lambda)$) if $f(\lambda)\prec g(\mu)$ (respectively, $f(\mu)\prec g(\lambda)$) for every $\mu\in I_\QQ$.

\begin{lma}
\label{prp:path:pathsubequivalence}
Let $S=(S,\prec)$ be a $\CatP$-semigroup, let $f$ be a path in $S$, and let $\lambda',\lambda\in I_\QQ$ satisfy $\lambda'<\lambda$.
Then there exists a path $h$ in $S$ such that $f(\lambda')\prec h\prec f(\lambda)$.
\end{lma}
\begin{proof}
Define $h\colon I_\QQ\to S$ by
\[
h(\gamma) :=f \big( \gamma\lambda+(1-\gamma)\lambda' \big),
\]
for $\gamma\in I_\QQ$.
Then $h$ is a path satisfying $f(\lambda')\prec h\prec f(\lambda)$, as desired.
\end{proof}

\begin{lma}
\label{prp:path:pathsuprema}
Let $S=(S,\prec)$ be a $\CatP$-semigroup.
Given a sequence $(f_n)_{n\geq 1}$ of paths in $S$, and given a sequence $(a_n)_{n\geq 1}$ in $S$ such that
\[
0\prec f_1 \prec a_1 \prec f_2\prec a_2 \prec f_3 \prec a_3  \cdots,
\]
there exists a path $h$ in $S$ such that $h(\tfrac{n}{n+1})=a_n$ for all $n\geq 1$.
\end{lma}
\begin{proof}
Define $h\colon I_\QQ\to S$ as follows:
\[
h(\lambda):=
\begin{cases}
f_n(\lambda), &\text{if } \lambda\in (\tfrac{n-1}{n},\tfrac{n}{n+1}) \\
a_n, &\text{if } \lambda=\tfrac{n}{n+1}
\end{cases}.
\]
It is easy to see that $h$ is a path and that $h(\tfrac{n}{n+1})=a_n$, as desired.
\end{proof}

\begin{lma}
\label{prp:path:pathO1}
Let $S$ be a $\CatP$-semigroup, and let $([f_n])_{n\geq 1}$ be an increasing sequence in $\tau(S)$.
Then there exists a strictly increasing sequence $(\lambda_m)_{m\geq 1}$ in $I_\QQ$ and a path $f$ in $S$ such that the following conditions hold:
\begin{enumerate}
\item
We have $\sup_m\lambda_m=1$.
\item
We have $f_n(\lambda_m)\prec f_l(\lambda_l)$, whenever $n,m<l$.
\item
We have $f(\tfrac{n}{n+1})=f_n(\lambda_n)$ for all $n\geq 1$.
\end{enumerate}
Moreover, if $f$ is a path in $S$ for which there exists a strictly increasing sequence $(\lambda_m)_{m\geq 1}$ in $I_\QQ$ satisfying conditions~(1), (2) and~(3) above, then $[f]=\sup_n [f_n]$ in $\tau(S)$.
In particular, $\tau(S)$ satisfies \axiomO{1}.
\end{lma}
\begin{proof}
The proof is divided in two parts.

We inductively find $\lambda_m\in I_\QQ$ and $h_m\in P(S)$ for $m\geq 1$ such that:
\begin{itemize}
\item[(a)]
$\lambda_{m-1}<\lambda_m$ and $\tfrac{m}{m+1}\leq\lambda_m$, for all $m$; and
\item[(b)]
$f_{n}(\lambda_{m-1})\prec h_m \prec  f_m(\lambda_m)$, for all $n<m$.
\end{itemize}
Set $\lambda_1 :=\frac{1}{2}$, and define the path $h_1$ by $h_1(\lambda)=f_1(\tfrac{\lambda}{2})$.
Note that $0\prec h_1\prec f_1(\lambda_1)$.

Assume we have chosen $\lambda_n$ and $h_n$ for all $n<m$.
For each $k=1,\ldots,m-1$, using that $f_k\precsim f_m$, we choose $\lambda_{m,k}\in I_\QQ$ such that $f_k(\lambda_{m-1})\prec f_m(\lambda_{m,k})$.
Let $\lambda_m'$ be the maximum of $\lambda_{m,1},\ldots,\lambda_{m,m-1},\tfrac{m}{m+1}$.
Choose $\lambda_m\in I_\QQ$ with $\lambda_m'<\lambda_m$.
Using \autoref{prp:path:pathsubequivalence}, we choose a path $h_m$ with $f_m(\lambda_m')\prec h_m\prec f_m(\lambda_m)$.

Note that in particular we have the following relations:
\[
0\prec h_1 \prec f_1(\lambda_1)\prec h_2\prec f_2(\lambda_2) \prec  h_3\prec  f_3(\lambda_3)\prec h_4\cdots
\]
Applying \autoref{prp:path:pathsuprema}, we choose $f\in P(S)$ with $f(\tfrac{n}{n+1}) =f_n(\lambda_n)$ for all $n\geq 1$.
Then it is easy to check that the sequence $(\lambda_m)_m$ and the path $f$ satisfy conditions~(1), (2) and~(3).

For the second part, let $(\lambda_m)_{m\geq 1}$ be a strictly increasing sequence in $I_\QQ$, and let $f\in P(S)$ satisfy (1), (2) and~(3).
We show that $[f]=\sup_n [f_n]$ in $\tau(S)$.

We first show that $[f_n]\leq [f]$ for each $n\geq 1$.
Fix $n\geq 1$.
To verify that $f_n\precsim f$, let $\lambda$ be an element in $I_\QQ$.
Use~(1) to choose $m$ with $n<m$ and $\lambda<\lambda_m$.
Using that $f_n$ is a path at the first step, using condition~(2) at the second step, and using~(3) at the last step, we obtain that
\[
f_n(\lambda) \prec f_n(\lambda_m) \prec f_{m+1}(\lambda_{m+1}) = f(\tfrac{m+1}{m+2}).
\]
Hence $f_n\precsim f$, as desired.

Conversely, let $g\in P(S)$ satisfy $f_n\precsim g$ for all $n\geq 1$.
To show that $f\precsim g$, take $\lambda\in I_\QQ$.
Choose $m$ such that $\lambda<\tfrac{m}{m+1}$.
Since $f_m\precsim g$, there exists $\mu\in I_\QQ$ such that $f_m(\lambda_m)\prec g(\mu)$.
Using this at the last step, using that $f$ is a path at the first step, and using condition~(3) at the second step, we get
\[
f(\lambda)\prec f(\tfrac{m}{m+1}) =f_m(\lambda_m) \prec g(\mu).
\]
This shows that $f\precsim g$, as desired.
\end{proof}

\begin{dfn}
\label{dfn:path:pathcutdown}
Let $S$ be a $\CatP$-semigroup, let $f\in P(S)$, and let $\varepsilon\in I_\QQ$.
We define $f_\varepsilon\colon I_\QQ\to S$ by
\[
f_\varepsilon(\lambda):=
\begin{cases}
f(\lambda-\varepsilon), &\text{if } \lambda>\varepsilon \\
0, &\text{otherwise}
\end{cases}.
\]
We will refer to $f_\varepsilon$ as the \emph{$\varepsilon$-cut down} of $f$.
\end{dfn}

\begin{rmk}
\label{rmk:path:pathcutdown}
It is easy to see that $f_\varepsilon$ is a path in $S$.
If $t$ is a real number, we write $t_+$ for $\max\{0,t\}$.
Then, under the convention that $f(0)=0$, we have $f_\varepsilon(\lambda)=f((\lambda-\varepsilon)_+)$ for all $\lambda\in I_\QQ$.
\end{rmk}

\begin{lma}
\label{prp:path:pathO2}
Let $S$ be a $\CatP$-semigroup, and let $f\in P(S)$.
Then $[f_\varepsilon]\ll[f_{\varepsilon'}]$ in $\tau(S)$, for every $\varepsilon',\varepsilon\in I_\QQ$ with $\varepsilon'<\varepsilon$.
Moreover, we have $[f]=\sup_{\varepsilon\in I_\QQ} [f_\varepsilon]$ in $\tau(S)$.
In particular, $\tau(S)$ satisfies \axiomO{2}.
\end{lma}
\begin{proof}
It is routine to check that $[f]=\sup_{\varepsilon}[f_{\varepsilon}]$.
Given $\varepsilon',\varepsilon\in I_\QQ$ with $\varepsilon'<\varepsilon$, note that  $f_{\varepsilon}=(f_{\varepsilon'})_{\varepsilon-\varepsilon'}$.
Thus it is enough to show that $[f_\varepsilon]\ll [f]$ for every $\varepsilon>0$.

Fix $\varepsilon>0$.
To show that $[f_\varepsilon]\ll [f]$, let $([g_n])_n$ be an increasing sequence in $\tau(S)$ with $[f]\leq \sup_n[g_n]$.
By \autoref{prp:path:pathO1}, there exists a path $h\in P(S)$ and an increasing sequence $(\lambda_n)_n$ in $I_\QQ$ such that $[h]=\sup_n[g_n]$, and such that $h(\tfrac{m}{m+1})= g_m(\lambda_m)$ for all $m\geq 1$.

Choose $m_0\geq 1$ with $\tfrac{1}{m_0}<\varepsilon$.
Since $f\precsim h$, there exists $\mu\in I_\QQ$ satisfying $f(1-\tfrac{1}{m_0})\prec h(\mu)$.
Choose $m_1\geq 1$ such that $\mu<\tfrac{m_1}{m_1+1}$.
Let us show that $f_\varepsilon\precsim g_{m_1}$.
For every $\lambda\in I_\QQ$, we have $\lambda-\varepsilon<1-\tfrac{1}{m_0}$.
Therefore, using that $f$ and $h$ are paths at the second and fourth step, respectively, and using that $f(1-\tfrac{1}{m_0})\prec h(\mu)$ at the third step, we obtain that
\[
f_\varepsilon(\lambda)
=f\big( (\lambda-\varepsilon)_+ \big)
\prec f(1-\tfrac{1}{m_0})
\prec h(\mu)
\prec h(\tfrac{m_1}{m_1+1})
= g_{m_1}(\lambda_{m_1}),
\]
for every $\lambda\in I_\QQ$.
This proves that $[f_\varepsilon]\leq [g_{m_1}]$, as desired.
\end{proof}

\begin{thm}
\label{prp:path:pathinCu}
Let $S$ be a $\CatP$-semigroup.
Then $\tau(S)$ is a $\CatCu$-semigroup.
\end{thm}
\begin{proof}
By \autoref{prp:path:tauispom}, \autoref{prp:path:pathO1} and \autoref{prp:path:pathO2}, $\tau(S)$ is a \pom{} that satisfies axioms \axiomO{1} and \axiomO{2}.
It remains to show that $\tau(S)$ satisfies \axiomO{3} and \axiomO{4}.

To verify \axiomO{3}, let $[f'],[f],[g'],[g]\in\tau(S)$ satisfy $[f']\ll[f]$ and $[g']\ll[g]$.
Using that $[f]=\sup_\varepsilon[f_\varepsilon]$, we can choose $\varepsilon_1\in I_\QQ$ such that $[f']\leq[f_{\varepsilon_1}]$.
Similarly we choose $\varepsilon_2$ for $[g]$.
Set $\varepsilon :=\min\{\varepsilon_1,\varepsilon_2\}$.
We then have $[f']\leq[f_\varepsilon]$ and $[g']\leq[g_\varepsilon]$.
Using that $f_\varepsilon+g_\varepsilon=(f+g)_\varepsilon$ at the third step, and using \autoref{prp:path:pathO2} at the fourth step, we deduce that
\[
[f']+[g']
\leq[f_\varepsilon]+[ g_\varepsilon]
=[f_\varepsilon+g_\varepsilon]
=[(f+g)_\varepsilon]
\ll[f+g]
=[f]+[g],
\]
which implies that $[f']+[g']\ll[f]+[g]$, as desired.

To prove \axiomO{4}, let $([f_n])_n$ and $([g_n])_n$ be two increasing sequences in $\tau(S)$.
It is clear that $\sup_n ([f_n]+[g_n])\leq\sup_n[f_n]+\sup_n[g_n]$.
Let us prove the converse inequality.

By \autoref{prp:path:pathO1}, there exist $f,g\in P(S)$ and increasing sequences $(\lambda_m)_m$ and $(\mu_m)_m$ in $I_\QQ$ such that $[f]=\sup_n[f_n]$ and $[g]=\sup_n[g_n]$, and such that
$f(\tfrac{m}{m+1})= f_m(\lambda_m)$ and $g(\tfrac{m}{m+1})= g_m(\mu_m)$ for all $m\in\NN$.
Given $\lambda\in I_\QQ$, choose $m\in\NN$ with $\lambda<\tfrac{m}{m+1}$.
Choose $\tilde\lambda\in I_\QQ$ such that $\lambda_m,\mu_m<\tilde\lambda$.
We deduce that
\begin{align*}
f(\lambda)+g(\lambda)
\prec (f+g)(\tfrac{m}{m+1})
= f_m(\lambda_m) + g_m(\mu_m)
\prec f_m(\tilde\lambda)+ g_m(\tilde\lambda).
\end{align*}
It follows that $[f]+[g]\leq\sup_n([f_n]+[g_n])$, as desired.
This verifies \axiomO{4}.
\end{proof}

The following result provides a useful criterion for compact containment in $\tau(S)$.

\begin{lma}
\label{prp:path:pathwaybelow}
Let $S$ be a $\CatP$-semigroup, and let $f',f$ be elements in $P(S)$.
Then $[f']\ll[f]$ in $\tau(S)$ if and only if there exists $\mu\in I_\QQ$ such that $f'\prec f(\mu)$.
\end{lma}
\begin{proof}
Assume that $[f']\ll[f]$.
Since $[f]=\sup_\varepsilon [ f_\varepsilon]$, there exists $\delta\in I_\QQ$ such that $[f']\leq [f_\delta]$.
Let us show that $\mu=1-\delta$ has the desired properties, that is, $f'\prec f(\mu)$.
Given $\lambda\in I_\QQ$, there is $\mu'\in I_\QQ$ with $f'(\lambda)\prec f_\delta(\mu')$.
Using that $(\mu'-\delta)_+<1-\delta=\mu$ at the last step, we deduce that
\[
f'(\lambda)\prec f_\delta(\mu')= f\big( (\mu'-\delta)_+ \big) \prec f(\mu).
\]

Conversely, suppose that there exists $\mu\in I_\QQ$ with $f'\prec f(\mu)$.
Then, for every $\mu'$ with $\mu<\mu'<1$ we have $[f']\leq [f_{1-\mu'}]\ll[f]$, as desired.
\end{proof}

\begin{lma}
\label{prp:path:tauOfMorphism}
Let $S$ and $T$ be $\CatP$-semigroups, and let $\alpha\colon S\to T$ be a $\CatP$-morphism.
Then, for every $f\in P(S)$, the map $\alpha\circ f\colon I_\QQ\to T$ belongs to $P(T)$.
Moreover, the induced map $\tau(\alpha)\colon\tau(S)\to\tau(T)$ given by
\[
\tau(\alpha)([f]) := [\alpha\circ f],
\]
for $f\in P(S)$, is a well-defined $\CatCu$-morphism.
\end{lma}
\begin{proof}
Given $f\in P(S)$, it is easy to see that $\alpha\circ f$ belongs to $P(T)$.
Moreover, given $f,g\in P(S)$ with $f\precsim g$ we have $\alpha\circ f\precsim \alpha\circ g$.
This shows that $\tau(\alpha)$ is well-defined and order-preserving.
It is also easy to see that $\tau(\alpha)$ preserves addition and the the zero element.

To show that $\tau(\alpha)$ preserves the way-below relation, let $f',f\in P(S)$ satisfy $[f']\ll[f]$ in $\tau(S)$.
By \autoref{prp:path:pathwaybelow}, there is $\mu\in I_\QQ$ with $f'\prec f(\mu)$.
Since $\alpha$ is a $\CatP$-morphism, we obtain that $\alpha\circ f'\prec(\alpha\circ f)(\mu)$.
A second usage of \autoref{prp:path:pathwaybelow} implies that $[\alpha\circ f']\ll [\alpha\circ f]$, as desired.

To show that $\tau(\alpha)$ preserves suprema of increasing sequences, let $([f_n])_n$ be such a sequence in $\tau(S)$.
By \autoref{prp:path:pathO1}, there exist $f\in P(S)$ and a strictly increasing sequence $(\lambda_m)_m$ in $I_\QQ$ such that the following conditions are satisfied:
\begin{enumerate}
\item
We have $\sup_m\lambda_m=1$.
\item
We have $f_n(\lambda_m)\prec f_l(\lambda_l)$, whenever $n,m<l$.
\item
We have $f(\tfrac{n}{n+1})=f_n(\lambda_n)$ for all $n\geq 1$.
\end{enumerate}
Further, for every $f\in P(S)$ satisfying these conditions, we have $[f]=\sup_n[f_n]$.

To show that $[\alpha\circ f]=\sup_n[\alpha\circ f_n]$, we verify that the path $\alpha\circ f$ and the sequence $(\lambda_m)_m$ satisfy the analogs of the above conditions with respect to the sequence $(\alpha\circ f_n)_n$.
Condition~(1) is unchanged.
To verify the analog of condition~(2), let $n,m<l$.
Since $\alpha$ is a $\CatP$-morphism, we have $(\alpha\circ f_n)(\lambda_m)\prec(\alpha\circ f_l)(\lambda_l)$, as desired.
The analog of~(3) holds, since
\[
(\alpha\circ f)(\tfrac{n}{n+1})
=(\alpha\circ f_n)(\lambda_n),
\]
for every $n\geq 1$.
Thus, the path $\alpha\circ f$ satisfies conditions~(1), (2) and~(3) for the sequence $(\alpha\circ f_n)_n$, which implies that $[\alpha\circ f]=\sup_n[\alpha\circ f_n]$.
Using this in the third step, we deduce that
\[
\tau(\alpha)(\sup_n[f_n])
=\tau(\alpha)([f])
=[\alpha\circ f]
=\sup_n[\alpha\circ f_n]
=\sup_n\tau(\alpha)([f_n]),
\]
as desired.
Altogether, we have that $\tau(\alpha)$ is a $\CatCu$-morphism.
\end{proof}

\begin{prp}
\label{prp:path:functoriality}
The $\tau$-construction defines a covariant functor $\tau\colon\CatP\to\CatCu$ by sending a $\CatP$-semigroup $S$ to the $\CatCu$-semigroup $\tau(S)$ (see \autoref{prp:path:pathinCu}), and by sending a $\CatP$-morphism $\alpha\colon S\to T$ to the $\CatCu$-morphism $\tau(\alpha)\colon\tau(S)\to\tau(T)$ (see \autoref{prp:path:tauOfMorphism}).
\end{prp}
\begin{proof}
It follow easily from the construction that $\tau(\id_{S})=\id_{\tau(S)}$ for every $\CatP$-semigroup $S$.
It is also straightforward to check that $\tau(\alpha\circ\beta)=\tau(\alpha)\circ\tau(\beta)$ for every pair of composable $\CatP$-morphisms $\alpha$ and $\beta$.
This shows that the $\tau$-construction defines a covariant functor, as claimed.
\end{proof}

Although the $\tau$-construction is a useful tool to derive $\CatCu$-semigroups from such simple objects as $\CatP$-semigroups, the next example shows that without additional care the $\tau$-construction may just produce a trivial object.

\begin{exa}
\label{exa:path:noPaths}
Consider $\NN=\{0,1,2\ldots\}$ with the usual structure as a monoid.
We define $\prec$ on $\NN$ by setting $k\prec l$ if $k<l$ or $k=l=0$.
It is easy to check that $(\NN,\prec)$ is a $\CatP$-semigroup, and that the only path in $P(\NN,\prec)$ is the constant path with value $0$.
It follows that $\tau(\NN,\prec)\cong\{0\}$.
\end{exa}

\section{The category \texorpdfstring{$\CatQ$}{Q}}
\label{sec:Q}

The category $\CatP$ introduced in the previous section, though useful in certain situations to construct $\CatCu$-semigroups from semigroups with very little structure, is too general to provide a nice categorical relation from $\CatP$ to  $\CatCu$. In this section we introduce a subcategory of $\CatP$, which we denote by $\CatQ$, where $\CatCu$ can be embedded as a full subcategory, and in such a way that the restriction of the $\tau$-construction from \autoref{sec:path} defines a coreflection $\CatQ\to\CatCu$;
see \autoref{thm:Q:coreflection}.

Recall the definition of an additive auxiliary relation from \autoref{dfn:prelim:auxRel}.

\begin{dfn}
\label{dfn:Q:CatQ}
A \emph{$\CatQ$-semigroup} is a \pom{} $S$ together with an additive, auxiliary relation $\prec$ on $S$ such that the following conditions are satisfied:
\begin{itemize}
\item[\axiomO{1}]
Every increasing sequence $(a_n)_n$ in $S$ has a supremum $\sup_n a_n$ in $S$.
\item[\axiomO{4}]
If $(a_n)_n$ and $(b_n)_n$ are increasing sequences in $S$, then $\sup_n(a_n+b_n)=\sup_n a_n+\sup_n b_n$.
\end{itemize}

Given $\CatQ$-semigroups $S$ and $T$, a \emph{$\CatQ$-morphism} from $S$ to $T$ is a map $S\to T$ that preserves addition, order, the zero element, the auxiliary relation and suprema of increasing sequences.
We denote the set of $\CatQ$-morphisms by $\CatQMor(S,T)$.
A \emph{generalized $\CatQ$-morphism} is a map that preserves addition, order, the zero element and suprema of increasing sequences.
We denote the set of generalized $\CatQ$-morphisms by $\CatQMor[S,T]$.

We let $\CatQ$ be the category whose objects are $\CatQ$-semigroups and whose morphisms are $\CatQ$-morphisms.
\end{dfn}

\begin{rmks}
\label{rmk:Q:CatQ}
(1)
Axioms \axiomO{1} and \axiomO{4} in \autoref{dfn:Q:CatQ} are the same as in \autoref{dfn:prelim:CatCu}.
A generalized $\CatQ$-morphism is a $\CatQ$-morphism if and only if it preserves the auxiliary relation.
Moreover, generalized $\CatQ$-morphisms are precisely the Scott continuous $\CatP$-morphisms.
(See \cite[Proposition~II-2.1, p.157]{GieHof+03Domains}.)

(2)
Let $S$, $T$ be $\CatQ$-semigroups.
The sets $\CatQMor[S,T]$ and $\CatQMor(S,T)$ of (generalized) $\CatQ$-morphism are \pom{s}, when equipped with the pointwise addition and order.
It is easy to see that $\CatQMor[S,T]$ satisfies \axiomO{1} and \axiomO{4}.
\end{rmks}

\begin{pgr}
\label{pgr:Q:inclusionCuQ}
We define a functor $\iota\colon\CatCu\to\CatQ$ as follows:
Given a $\CatCu$-semigroup $S$, the (sequential) way-below relation $\ll$ is an additive auxiliary relation on $S$.
It follows that $(S,\ll)$ is a $\CatQ$-semigroup, and we let $\iota$ map $S$ to $(S,\ll)$.

Further, given $\CatCu$-semigroups $S$ and $T$, a map $\varphi\colon S\to T$ is a $\CatCu$-morphism if and only if $\varphi\colon(S,\ll)\to(T,\ll)$ is a $\CatQ$-morphism.
We let $\iota$ map a $\CatCu$-morphism to itself, considered as a $\CatQ$-morphism.
This defines a functor from $\CatCu$ to $\CatQ$.
\end{pgr}

\begin{prp}
\label{prp:Q:fullCuQ}
The functor $\iota\colon\CatCu\to\CatQ$ from \autoref{pgr:Q:inclusionCuQ} embeds $\CatCu$ as a full subcategory of $\CatQ$.
\end{prp}

Every $\CatQ$-semigroup can be considered as a $\CatP$-semigroup by forgetting its partial order.
Therefore, if $S$ is a $\CatQ$-semigroup with auxiliary relation $\prec$, then a \emph{path} in $S$ is a map $f\colon I_\QQ\to S$ such that $f(\lambda')\prec f(\lambda)$ whenever $\lambda',\lambda\in I_\QQ$ satisfy $\lambda'<\lambda$;
see \autoref{dfn:path:path} and \autoref{ntn:path:standardPathType}.
Recall that $P(S)$ denotes the set of paths in $S$.

\begin{dfn}
\label{dfn:Q:endpoint}
Let $S$ be a $\CatQ$-semigroup, and let $f\in P(S)$.
We define the \emph{endpoint} of $f$, denoted by $f(1)$, as $f(1) := \sup_{\lambda\in I_\QQ} f(\lambda)$.
\end{dfn}

\begin{prp}
\label{prp:Q:endpoint}
Let $S$ be a $\CatQ$-semigroup, and let $f,g\in P(S)$.
Then:
\begin{enumerate}
\item
We have $(f+g)(1)=f(1)+g(1)$ in $S$.
\item
If $f\precsim g$, then $f(1)\leq g(1)$ in $S$.
\item
If $[f]\ll[g]$ in $\tau(S)$, then $f(1)\prec g(1)$.
\item
If $([f_n])_n$ is an increasing sequence in $\tau(S)$ and $[f]=\sup_n [f_n]$, then $f(1)=\sup_n f_n(1)$ in $S$.
\end{enumerate}
\end{prp}
\begin{proof}
(1):
This is a consequence of the fact that $S$ satisfies \axiomO{4}.

(2):
Given $\lambda\in I_\QQ$, using that $f\precsim g$, there is $\mu\in I_\QQ$ with $f(\lambda)\prec g(\mu)\leq g(1)$.
Taking the supremum over $\lambda$, we obtain that $f(1)\leq g(1)$.

(3):
Assuming $[f]\ll[g]$, we use \autoref{prp:path:pathwaybelow} to choose $\mu\in I_\QQ$ with $f\prec g(\mu)$.
Then $f(1)\leq g(\mu)\prec g(1)$.

(4):
Let $([f_n])_n$ be an increasing sequence in $\tau(S)$, and let $[f]=\sup_n [f_n]$.
By (2), the endpoint of a path only depends on its equivalence class with respect to the relation $\sim$ from \autoref{dfn:path:path}.

By \autoref{prp:path:pathO1}, there are $f'\in P(S)$ and an increasing sequence $(\lambda_m)_m$ in $I_\QQ$ such that $\sup_m\lambda_m=1$ and $[f']=\sup_n[f_n]$, and such that $f'(\tfrac{n}{n+1})=f_n(\lambda_n)$ for all $n\in\NN$.
Using that $f'\sim f$ at the first step, and using the above property of $f'$ at the fourth step we obtain that
\[
f(1)
=f'(1)
=\sup_{\lambda\in I_\QQ} f'(\lambda)
=\sup_n f'(\tfrac{n}{n+1})
=\sup_n f_n(\lambda_n)
\leq\sup_n f_n(1).
\]
For each $n$, we have $f_n\precsim f$ and therefore $f_n(1)\leq f(1)$ by~(2).
It follows that $\sup_n f_n(1)\leq f(1)$, and therefore $f(1)=\sup_n f_n(1)$, as desired.
\end{proof}

By \autoref{prp:Q:endpoint}, the endpoint of a path only depends on the equivalence class in $\tau(S)$.
Therefore, the following definition makes sense.

\begin{dfn}
\label{dfn:Q:endpointMap}
Let $S$ be a $\CatQ$-semigroup.
We define a map $\varphi_S\colon\tau(S)\to S$ by
\[
\varphi_S([f]) := f(1),
\]
for all $f\in P(S)$.
We refer to $\varphi_S$ as the \emph{endpoint map}.
\end{dfn}

\begin{prp}
\label{prp:Q:endpointnatural}
Let $S$ be a $\CatQ$-semigroup.
Then the endpoint map $\varphi_S\colon\tau(S)\to S$ is a well-defined $\CatQ$-morphism (when considering $\tau(S)$ as a $\CatQ$-morphism via the inclusion functor $\iota$ from \autoref{pgr:Q:inclusionCuQ}.)

Moreover, the endpoint map is natural in the sense that $\alpha\circ\varphi_S=\varphi_T\circ\tau(\alpha)$ for every $\CatQ$-morphism $\alpha\colon S\to T$ between $\CatQ$-semigroups $S$ and $T$.
This means that the following diagram commutes:
\[
\xymatrix{
\tau(S) \ar[r]^-{\varphi_S}\ar[d]_{\tau(\alpha)} &  S\ar[d]^{\alpha} & \\
\tau(T)\ar[r]_-{\varphi_T} & T
}
\]
\end{prp}
\begin{proof}
It follows directly from \autoref{prp:Q:endpoint} that $\varphi_S$ is a well-defined $\CatQ$-mor\-phism.
To show the commutativity of the diagram, let $f\in P(S)$.
Using that $\alpha$ preserves suprema of increasing sequences at the second step, we deduce that
\[
\alpha\left( \varphi_S([f]) \right)
=\alpha\left( \sup_{\lambda\in I_\QQ}f(\lambda) \right)
=\sup_{\lambda\in I_\QQ}\alpha(f(\lambda))
=\varphi_T([\alpha\circ f])
=\varphi_T(\tau(\alpha)([f])),
\]
as desired.
\end{proof}

\begin{rmk}
The naturality of the endpoint map as formulated in \autoref{prp:Q:endpointnatural} means precisely that the $\CatQ$-morphisms $\varphi_S$, for $S$ ranging over the objects in $\CatQ$, form the components of a natural transformation from $\iota\circ\tau$ to the identity functor on $\CatQ$.
\end{rmk}

In general, the endpoint map is neither surjective nor injective;
see Examples~\ref{exa:Q:noPaths} and~\ref{exa:Q:M}.
We now show that $\varphi_S$ is an order-isomorphism if (and only if) $S$ is a $\CatCu$-semigroup.

\begin{prp}
\label{prp:Q:phiisisoforCu}
Let $S$ be a $\CatCu$-semigroup, considered as a $\CatQ$-semigroup $(S,\ll)$.
Then the endpoint map $\varphi_S\colon{\tau(S,\ll)}\to S$ is an order-isomorphism.
\end{prp}
\begin{proof}
We first prove that $\varphi_S$ is an order-embedding.
Let $[f],[g]\in\tau(S,\ll)$ satisfy $\varphi_S([f])\leq \varphi_S([g])$.
Then, by definition, $\sup_{\mu} f(\mu) \leq \sup_{\mu} g(\mu)$.
To show that $f\precsim g$, let $\lambda\in I_\QQ$.
Choose $\tilde\lambda\in I_\QQ$ with $\lambda<\tilde\lambda$.
We deduce that
\[
f(\lambda) \ll f(\tilde\lambda) \leq \sup_{\mu} f(\mu) \leq \sup_{\mu} g(\mu).
\]
Therefore, there exists $\mu\in I_\QQ$ such that $f(\lambda)\leq g(\mu)$.
Choose $\tilde{\mu}\in I_\QQ$ with $\mu<\tilde{\mu}$.
Then $f(\lambda)\ll g(\tilde\mu)$.
This implies that $f\precsim g$ and thus $[f]\leq [g]$, as desired.

To show that $\varphi_S$ is surjective, let $s\in S$.
Choose a $\ll$-increasing chain $(s_\lambda)_{\lambda\in (0,1)}$ as in \autoref{prp:prelim:ctsO2}.
In particular, we have $s=\sup_\lambda s_\lambda$, and $s_{\lambda'}\ll s_\lambda$ whenever $\lambda',\lambda\in I_\QQ$ satisfy $\lambda'<\lambda$.
Thus, if we define $f\colon I_\QQ\to S$ by $f(\lambda):=s_\lambda$, for $\lambda\in I_\QQ$, then $f$ belongs to $P(S,\ll)$.
By construction, $\varphi_S([f])=s$, as desired.
\end{proof}

Given $\CatQ$-semigroups $S$ and $T$, recall that we equip the set of $\CatQ$-morphisms $\CatQMor(S,T)$ with pointwise order and addition;
see \autoref{rmk:Q:CatQ}.

\begin{prp}
\label{prp:Q:univpropertytau}
Let $T$ be a $\CatCu$-semigroup, let $S$ be a $\CatQ$-semigroup, and let $\varphi_S\colon\tau(S)\to S$ be the endpoint map from \autoref{dfn:Q:endpointMap}.
Then:
\begin{enumerate}
\item
For every $\CatQ$-morphism $\alpha\colon T\to S$ there exists a $\CatCu$-morphism $\bar{\alpha}\colon T\to\tau(S)$ such that $\varphi_S\circ\bar{\alpha}=\alpha$.
\item
We have $\varphi_S\circ\beta\leq\varphi_S\circ\gamma$ if and only if $\beta\leq\gamma$, for any pair of $\CatCu$-morphisms $\beta,\gamma\colon T\to\tau(S)$.
\end{enumerate}
Statement~(1) means that for every $\alpha$ one can find $\bar{\alpha}$ making the following diagram commute:
\[
\xymatrix{
\tau(S) \ar[r]^-{\varphi_S} & S \\
& T \ar@{-->}[ul]^{\bar\alpha} \ar[u]_{\alpha}.
}
\]
\end{prp}
\begin{proof}
To show~(1), let $\alpha$ be given.
Since $T$ is a $\CatCu$-semigroup, it follows from \autoref{prp:Q:phiisisoforCu} that $\varphi_T\colon\tau(T,\ll)\to T$ is an order-isomorphism.
Set $\bar\alpha:=\tau(\alpha)\circ\varphi_T^{-1}$, which is clearly a $\CatCu$-morphism. By \autoref{prp:Q:endpointnatural}, we have $\varphi_S\circ\tau(\alpha)=\alpha\circ\varphi_T$. It follows that $\varphi_S\circ\bar\alpha=\alpha$.
The maps are shown in the following diagram:
\[
\xymatrix{
\tau(S) \ar[r]^-{\varphi_S} & S   \\
\tau(T,\ll)\ar[u]^{\tau(\alpha)}\ar[r]_-{\varphi_T}
& T. \ar[u]_\alpha\ar@{-->}@/^1pc/[ul]_{\bar\alpha}
}
\]

To show~(2), let $\beta,\gamma\colon T\to \tau(S)$ be $\CatCu$-morphisms.
It is clear that $\beta\leq\gamma$ implies that $\varphi_S\circ\beta\leq\varphi_S\circ\gamma$.
Thus let us assume that $\varphi_S\circ\beta\leq\varphi_S\circ\gamma$.

To show that $\beta\leq\gamma$, let $t\in T$.
Using that $T$ satisfies \axiomO{2}, choose a $\ll$-increasing sequence $(t_n)_n$ in $T$ with supremum $t$.
Fix $n$, and choose paths $f_n,g_n,g\in P(S)$ with $\beta(t_n)=[f_n]$, and $\gamma(t_n)=[g_n]$, and $\gamma(t)=[g]$.
Since $\gamma$ preserves the way-below relation, we have $[g_n]\ll[g]$ in $\tau(S)$.
By \autoref{prp:path:pathwaybelow}, we can choose $\mu\in I_\QQ$ such that $g_n(\lambda)\prec g(\mu)$ for all $\lambda\in I_\QQ$.
Passing to the supremum over $\lambda$, we obtain that $g_n(1)\leq g(\mu)$.
Using this at the last step, and using the assumption that $\varphi_S\circ\beta\leq\varphi_S\circ\gamma$ at the second step, we deduce that
\[
f_n(\lambda) \leq f_n(1)
= \varphi_S(\beta(t_n))
\leq \varphi_S(\gamma(t_n))
= g_n(1)
\leq g(\mu),
\]
for every $\lambda\in I_\QQ$.
By definition, we have that $f_n\precsim g$, and hence $\beta(t_n)\leq\gamma(t)$.

Using that $\beta$ preserves suprema of increasing sequences at the second step, and using the above observation $\beta(t_n)\leq\gamma(t)$ for each $n$ at the last step, we deduce that
\[
\beta(t)
=\beta\left( \sup_n t_n \right)
=\sup_n\beta(t_n)
\leq\gamma(t),
\]
as desired.
\end{proof}

\begin{thm}
\label{thm:Q:coreflection}
The category $\CatCu$ is a coreflective, full subcategory of $\CatQ$;
the functor $\tau\colon\CatQ\to\CatCu$ is a right adjoint to the inclusion functor $\iota\colon\CatCu\to\CatQ$ from \autoref{pgr:Q:inclusionCuQ}.

More precisely, let $S$ be a $\CatQ$-semigroup, let $\varphi_S\colon\tau(S)\to S$ be the endpoint map from \autoref{dfn:Q:endpointMap}, and let $T$ be a $\CatCu$-semigroup.
Then the assignment 
that sends a $\CatCu$-morphism $\beta\colon T\to\tau(S)$ to the $\CatQ$-morphism $\varphi_S\circ\beta$ defines a natural bijection
\[
\CatCuMor\big( T,\tau(S) \big) \cong \CatQMor(T,S),
\]
which respects the structure of the morphism sets as positively ordered monoids.
\end{thm}
\begin{proof}
Let us denote the assignment from the statement by $\Phi\colon\CatCuMor(T,\tau(S)) \to \CatQMor(T,S)$.
Then $\Phi$ is well-defined since $\varphi_S$ is a $\CatQ$-morphism by \autoref{prp:Q:endpointnatural}.
Statement~(1) in \autoref{prp:Q:univpropertytau} means exactly that $\Phi$ is surjective.
Further, statement~(2) in \autoref{prp:Q:univpropertytau} shows that $\Phi$ is an order-embedding.
Thus, $\Phi$ is an order-isomorphism, and in particular bijective.
This shows that $\tau$ is right adjoint to $\iota$, as desired.
\end{proof}

We now consider examples of $\CatQ$-semigroups and their associated endpoint maps.
In \autoref{exa:Q:M} we introduce two important $\CatCu$-semigroups that are obtained by using the $\tau$-construction.
We denote these $\CatCu$-semigroups by $M_1$ and $M_\infty$ since they turn out to be the Cuntz semigroups of $\mathrm{II}_1$- and $\mathrm{II}_\infty$-factors, respectively;
see \autoref{prp:Q:M}.

\begin{exa}
\label{exa:Q:noPaths}
Consider $\NNbar=\{0,1,2,\ldots,\infty\}$ with $\prec$ as in \autoref{exa:path:noPaths}.
Then $\tau(\NNbar,\prec)=\{0\}$, which shows that the endpoint map need not be surjective.
\end{exa}

Recall that, given a $\CatCu$-semigroup $S$ and $a\in S$, we say that $a$ is \emph{soft} if for every $a'\in S$ with $a'\ll a$ we have $a'<_s a$, that is, there exists $k\in\NN$ with $(k+1)a'\leq ka$; see \cite[Definition~5.3.1]{AntPerThi14arX:TensorProdCu}. We denote the set of soft elements in $S$ by $S_\txtSoft$.

\begin{exa}
\label{exa:Q:M}
Consider $\PPbar:=[0,\infty]$, with its usual structure as a \pom{}.
We define two relations $\prec_1$ and $\prec_\infty$ on $\PPbar$ as follows:
given $a,b\in\PPbar$ we set $a\prec_1 b$ if and only if $a<\infty$ and $a\leq b$;
and we set $a\prec_\infty b$ if and only if $a\leq b$.
It is easy to check that $(\PPbar,\prec_1)$ and $(\PPbar,\prec_\infty)$ are $\CatQ$-semigroups.
We set
\[
M_1 := \tau\big( \PPbar,\prec_1 \big),
\quad\text{ and }\quad
M_\infty := \tau\big( \PPbar,\prec_\infty \big).
\]

Let us compute the precise structure of $M_1$ and $M_\infty$.
For the most part, the argument is the same in both cases, and we use $\prec_\ast$ to stand for either $\prec_1$ or $\prec_\infty$.
Recall that $P(\PPbar,\prec_\ast)$ is the set of $\prec_\ast$-increasing map $f\colon I_\QQ\to \PPbar$.
Given a path $f$, we let $f(1)$ denote the endpoint, that is, $f(1)=\sup_{\lambda\in I_\QQ}f(\lambda)$;
see \autoref{dfn:Q:endpoint}.

Let $f,g\in P(\PPbar,\prec_\ast)$.
If $f\precsim g$, then $f(1)\leq g(1)$, by \autoref{prp:Q:endpoint} (2).
Conversely, if $ f(1)<g(1)$, then it is easy to deduce that $f\precsim g$.
In fact, it is clear that the equivalence class of a path only depends on its definition in $(1-\varepsilon,1)\cap I_\QQ$, for some $\varepsilon>0$.
Therefore, all eventually constant paths with the same endpoint are equivalent and they majorize any path with the same endpoint.
Furthermore, two paths with equal endpoint which are not eventually constant are in fact equivalent.

Thus, for every $a\in(0,\infty)$ there are exactly two equivalence classes of paths with endpoint $a$:
the classes $[f_a']$ and $[f_a]$ with $f_a'$ and $f_a$ given by $f_a'(\lambda)=\lambda a$ and $f_a(\lambda)=a$, for $\lambda\in I_\QQ$. The endpoints $0$ and $\infty$ are particular:
The only path with endpoint $0$ is the constant path $f_0$ with value $0$.

The only difference between $\prec_1$ and $\prec_\infty$ appears now for paths with endpoint $\infty$.
There is no $\prec_1$-increasing path that is (eventually) constant with value $\infty$.
Therefore, all paths in $P(\PPbar,\prec_1)$ with endpoint $\infty$ are equivalent to $f_\infty'$ given by $f_\infty'(\lambda)=\tfrac{1}{1-\lambda}$.
On the other hand, $P(\PPbar,\prec_\infty)$ also contains the constant path $f_\infty$ with value $\infty$.
We obtain that
\[
M_1 = \big\{ [f_0],[f_a'],[f_a],[f_\infty'] : a\in(0,\infty) \big\},
\quad\text{ and }\quad
M_\infty = M_1 \cup \big\{ [f_\infty] \big\}.
\]
Thus, $M_1$ and $M_\infty$ differ only in that $M_\infty$ contains an additional infinite element.
It is easy to see that the natural map $M_1\to M_\infty$ is an additive order-embedding.
Hence, it suffices to describe order and addition in $M_\infty$.
We have $[f_0]\leq[f_a']\leq[f_a]\leq[f_\infty']\leq[f_\infty]$ for $a\in(0,\infty)$.
Further, for $a,b\in(0,\infty)$ we have $[f_a']\leq [f_b]$ if and only if $a\leq b$; and $[f_a]\leq[f_b']$ if and only if $a<b$.
We have $[f_\infty']<[f_\infty]$.

It is straightforward to check that the addition in $M_\infty$ is given by
\[
[f_a]+[f_b] = [f_{a+b}],
\quad\text{ and }\quad
[f_a']+[f_b] = [f_a']+[f_b'] = [f_{a+b}'],
\]
for $a\in [0,\infty]$ and $b\in[0,\infty)$.
We have that $[f_\infty']+[f_\infty]=[f_\infty]$.

Abusing notation, we use $a'$ and $a$ to denote $[f_a']$ and $[f_a]$ in $M_\infty$.
Further, we use $0$ to denote the classes of $f_0$. Now, the compact elements in $M_1$ are $0$ and $a$ for $a\in(0,\infty)$.
The soft elements in $M_1$ are $0$ and $a'$ for $a\in(0,\infty]$.
The additional element $\infty$ in $M_\infty$ is both soft and compact.

The endpoint map $M_1\to\PPbar$ is not injective since it sends both $[f_a]$ and $[f_a']$ to $a$, for every $a\in(0,\infty)$.
Analogously, the endpoint map $M_\infty\to\PPbar$ is not injective.
\end{exa}

\begin{prp}
\label{prp:Q:M}
We have $\Cu(M)\cong M_1$ for every $\mathrm{II}_1$-factor $M$;
and we have $\Cu(N)\cong M_\infty$ for every $\mathrm{II}_\infty$-factor $N$.
\end{prp}
\begin{proof}
Let $M$ be a $\mathrm{II}_1$-factor $M$, let $\tau\colon M_+\to[0,1]$ denote its unique tracial state, and let $\tilde{\tau}\colon (M\otimes\KK)_+\to[0,\infty]$ denote the unique extension to a tracial weight on the stabilization. It is known that $V(M)$ is isomorphic to $[0,\infty)$, with the usual structure as a \pom, via $[p]\mapsto\tilde{\tau}(p)$.

Recall that a countably generated interval in a \pom{} is a nonempty, upward directed, order-hereditary subset that contains a countable cofinal subset.
By \cite[Theorem~6.4]{AntBosPer11CompletionsCu}, the Cuntz semigroup of a $\sigma$-unital \ca{} $A$ with real rank zero can be computed as $\Cu(A)\cong\Lambda_\sigma(V(A))$, the set of countably-generated intervals in $V(A)$. 

Now, the countably generated intervals in $[0,\infty)$ are given as:
$I_0:=\{0\}$; $I_a':=[0,a)$ and $I_a:=[0,a]$, for $a\in(0,\infty)$; and $I_\infty':=[0,\infty)$.
We obtain an order-isomorphism $\Lambda_\sigma([0,\infty))\cong M_1$ by mapping $I_a$ to $[f_a]$, for $a\in[0,\infty)$, and by mapping $I_a'$ to $[f_a']$, for $a\in(0,\infty]$.
Together, we obtain order-isomorphisms:
\[
\Cu(M) \cong \Lambda_\sigma(V(M)) \cong \Lambda_\sigma([0,\infty)) \cong M_1.
\]

For a $\mathrm{II}_\infty$-factor $N$, the argument runs analogous to the $\mathrm{II}_1$-case, with the difference that $N$ contains infinite projections. We thus have $V(N)\cong[0,\infty]$, and therefore $\Cu(N) \cong \Lambda_\sigma(V(N)) \cong \Lambda_\sigma([0,\infty]) \cong M_\infty$.
\end{proof}

\begin{dfn}
\label{dfn:Q:auxGenMor}
Let $S$ and $T$ be $\CatQ$-semigroups.
We define a binary relation $\prec$ on the set of generalized $\CatQ$-morphisms $\CatQMor[S,T]$ by setting $\varphi\prec\psi$ if and only $\varphi\leq \psi$ and $\varphi(a')\prec \psi(a)$ for all $a',a\in S$ with $a'\prec a$.
\end{dfn}

\begin{prp}
\label{prp:Q:genMorInQ}
Let $S$ and $T$ be $\CatQ$-semigroups.
Then the relation $\prec$ on $\CatQMor[S,T]$, as defined in \autoref{dfn:Q:auxGenMor}, is an auxiliary relation.
Moreover, $(\CatQMor[S,T],\prec)$ is a $\CatQ$-semigroup.
\end{prp}
\begin{proof}
Since addition and order in $\CatQMor[S,T]$ are defined pointwise, it is easy to verify that $\CatQMor[S,T]$ is a \pom{}.
Given an increasing sequence $(\varphi_n)_n$ in $\CatQMor[S,T]$, let $\varphi\colon S\to T$ be the pointwise supremum, that is, $\varphi(s):=\sup_n\varphi_n(s)$, for $s\in S$.
Then clearly $\varphi$ is a generalized $\CatQ$-morphism and $\sup_n\varphi_n=\varphi$ in $\CatQMor[S,T]$.
Thus, $\CatQMor[S,T]$ satisfies \axiomO{1}.
It is also clear that taking suprema is compatible with addition and hence $\CatQMor[S,T]$ also satisfies \axiomO{4}.

Next, note that $\prec$ is an auxiliary relation on $\CatQMor[S,T]$. It is also easy to verify that $\prec$ is additive.
Therefore, $(\CatQMor[S,T],\prec)$ is a $\CatQ$-semigroup.
\end{proof}

Next, we define bimorphisms in the category $\CatQ$ analogous to the definition of $\CatCu$-bimorphisms;
see \autoref{dfn:prelim:CatCuBimor}.
Recall the definition of $\CatPom$-bimorphisms from \autoref{pgr:prelim:CatPom}.

\begin{dfn}
\label{dfn:Q:CatQBimor}
Let $S,T$ and $P$ be $\CatQ$-semigroups, and let $\varphi\colon S\times T\to P$ be a $\CatPom$-bimorphism.
We say that $\varphi$ is a \emph{$\CatQ$-bimorphism} if it satisfies the following conditions:
\begin{enumerate}
\item
We have that $\sup_k\varphi(a_k,b_k)=\varphi(\sup_k a_k, \sup_k b_k)$, for every increasing sequences $(a_k)_k$ in $S$ and $(b_k)_k$ in $T$.
\item
If $a',a\in S$ and $b',b\in T$ satisfy $a'\prec a$ and $b'\prec b$, then $\varphi(a',b')\prec\varphi(a,b)$.
\end{enumerate}
We denote the set of $\CatQ$-bimorphisms by $\CatQBimor(S\times T,P)$.
\end{dfn}

Given $\CatQ$-semigroups $S,T$ and $P$, we equip $\CatQBimor(S\times T,P)$ with pointwise order and addition, giving it the structure of a \pom.
Similarly, we consider the set of $\CatQ$-morphisms between two $\CatQ$-semigroups as a \pom{} with the pointwise order and addition.

The proof of the following result follows straightforward from the definition of $\CatQ$-bimorphisms and is therefore omitted.

\begin{lma}
\label{prp:Q:sliceOfBimor}
Let $S,T$ and $P$ be $\CatQ$-semigroups, and let $\varphi\colon S\times T\to P$ be a $\CatQ$-bimorphism.
For each $a\in S$, define $\varphi_a\colon T\to P$ by $\varphi_a(b)=\varphi(a,b)$.
Then $\varphi_s$ belongs to $\CatQMor[T,P]$.
Moreover, if $a',a\in S$ satisfy $a'\prec a$, then $\varphi_{a'}\prec\varphi_a$.
\end{lma}

\begin{ntn}
\label{ntn:Q:sliceOfBimor}
Let $S,T$ and $P$ be $\CatQ$-semigroups, and let $\varphi\colon S\times T\to P$ be a $\CatQ$-bimorphism.
Using \autoref{prp:Q:sliceOfBimor} we may define a map $\tilde{\varphi}\colon S\to\CatQMor[T,P]$ by $\tilde{\varphi}(a)=\varphi_a$, for $a\in S$, which belongs to $\CatQMor(S,\CatQMor[T,P])$.
\end{ntn}

\begin{thm}
\label{prp:Q:ihom}
Let $S,T$ and $P$ be $\CatQ$-semigroups.
Then:
\begin{enumerate}
\item
For every $\CatQ$-morphism $\alpha\colon S\to\CatQMor[T,P]$ there exists a $\CatQ$-bimorphism $\varphi\colon S\times T\to P$ such that $\alpha=\tilde{\varphi}$.
\item
If $\varphi,\psi\colon S\times T\to P$ are $\CatQ$-bimorphisms, then $\varphi\leq\psi$ if and only if $\tilde{\varphi}\leq\tilde{\psi}$.
\end{enumerate}

Thus, the assignment $\Phi$ that sends a $\CatQ$-bimorphism $\varphi\colon S\times T\to P$ to the $\CatQ$-morphism $\tilde{\varphi}\colon S\to\CatQMor[T,P]$ defines a natural bijection
\[
\CatQBimor\big( S\times T,P \big) \cong \CatQMor\big( S,\CatQMor[T,P] \big),
\]
which respects the structure of the (bi)morphism sets as \pom{s}.
\end{thm}
\begin{proof}
To verify~(1), let  $\alpha\colon S\to\CatQMor[T,P]$ be a $\CatQ$-morphism.
Define $\varphi\colon S\times T\to P$ by $\varphi(s,t)=\alpha(s)(t)$.
It is straightforward to check that $\varphi$ is a $\CatQ$-bimorphism satisfying $\alpha=\tilde{\varphi}$, as desired.
Statement~(2) is also easily verified.
It follows that $\Phi$ is an order-isomorphism, and hence a bijection.
It is also clear that $\Phi$ is additive and preserves the zero element.
\end{proof}

\begin{lma}
\label{prp:Q:functorial}
Let $S_1$, $S_2$ and $T$ be $\CatQ$-semigroups, and let $\alpha\colon S_1\to S_2$ be a (generalized) $\CatQ$-morphism.
Then the map $\alpha^*\colon\CatQMor[S_2,T]\to\CatQMor[S_1,T]$ given by $\alpha^*(f) := f\circ\alpha$,
for $f\in\CatQMor[S_2,T]$, is a (generalized) $\CatQ$-morphism.

Analogously, given $\CatQ$-semigroups $S$, $T_1$ and $T_2$, and given a (generalized) $\CatQ$-morphism $\beta\colon T_1\to T_2$, the map $\beta_*\colon\CatQMor[S,T_1]\to\CatQMor[S,T_2]$ defined by $\beta_*(f) := \beta\circ f$, for $f\in\CatQMor[S,T_1]$, is a (generalized) $\CatQ$-morphism.
\end{lma}
\begin{proof}
It is straightforward to check that $\alpha^*$ and $\beta_*$ are generalized $\CatQ$-morphisms.
Assume that $\alpha$ is a $\CatQ$-morphism.
To show that $\alpha^*$ preserves the auxiliary relation, let $f_1,f_2\in\CatQMor[S_2,T]$ satisfy $f_1\prec f_2$.
To show that $\alpha^*(f_1)\prec\alpha^*(f_2)$, let $a',a\in S$ satisfy $a'\prec a$.
Since $\alpha$ preserves the auxiliary relation, we have $\alpha(a')\prec\alpha(a)$.
Using that $f_1\prec f_2$ at the second step, we deduce that
\[
\alpha^*(f_1)(a')
= f_1(\alpha(a'))
\prec f_2(\alpha(a))
= \alpha^*(f_2)(a),
\]
as desired.
Analogously, one shows that $\beta_*$ preserves the auxiliary relation whenever $\beta$ does.
\end{proof}

\begin{pgr}
\label{pgr:Q:ihom_functor}
Let $T$ be a $\CatQ$-semigroup.
We let $\CatQMor[\freeVar,T]\colon\CatQ\to\CatQ$ be the contravariant functor that sends a $\CatQ$-semigroup $S$ to the $\CatQ$-semigroup $\CatQMor[S,T]$ (see \autoref{prp:Q:genMorInQ}), and that sends a $\CatQ$-morphism $\alpha\colon S_1\to S_2$ to the $\CatQ$-morphism $\alpha^*\colon\CatQMor[S_2,T]\to\CatQMor[S_1,T]$ as in \autoref{prp:Q:functorial}.

Analogously, we obtain a covariant functor $\CatQMor[S,\freeVar]\colon\CatQ\to\CatQ$ for every $\CatQ$-sem\-i\-group $S$.
Thus, we obtain a bifunctor
\[
\CatQMor[\freeVar,\freeVar]\colon\CatQ\times\CatQ\to\CatQ.
\]
\end{pgr}

\section{Abstract bivariant Cuntz semigroups}
\label{sec:bivarCu}
In this section, we use the $\tau$-construction developed in Sections~\ref{sec:path} and~\ref{sec:Q} to prove that $\CatCu$ is a \emph{closed} symmetric monoidal category.

\subsection{Construction of abstract bivariant Cuntz semigroups}
\label{sec:bivarCu:construction}

Recall the notion of a \emph{generalized \CuMor{}} (see \autoref{dfn:prelim:CatCu}), and that the set of generalized \CuMor{s} $S\to T$ is denoted by $\CatCuMor[S,T]$.
We equip this set with pointwise order and addition, giving it a natural structure as a \pom.

From \autoref{pgr:Q:inclusionCuQ}, there is functor $\iota\colon\CatCu\to\CatQ$ that embeds $\CatCu$ as a full subcategory of $\CatQ$. This is given by
considering a \CuSgp{} $S$ as a $\CatQ$-semigroup for the auxiliary relation $\ll$.

In \autoref{dfn:Q:auxGenMor} we introduced an auxiliary relation on the set of generalized $\CatQ$-morphisms, giving itself the structure of a $\CatQ$-semigroup;
see \autoref{prp:Q:genMorInQ}.
Let us transfer this definition to the setting of $\CatCu$-semigroups.

\begin{dfn}
\label{dfn:bivarCu:auxGenMor}
Let $S$ and $T$ be \CuSgp{s}.
We define a binary relation $\prec$ on the set of generalized \CuMor{s} $\CatCuMor[S,T]$ by setting $\varphi\prec \psi$ if and only $\varphi(a')\ll\psi(a)$ for all $a',a\in S$ with $a'\ll a$.
\end{dfn}

\begin{rmks}
(1)
The auxiliary relation $\prec$ on the set of generalized \CuMor{s} was already considered in \cite[6.2.6]{AntPerThi14arX:TensorProdCu}. It is easy to verify that, for $\varphi,\psi\in \CatCuMor[S,T]$, the relation $\varphi\prec \psi$ as defined in \autoref{dfn:bivarCu:auxGenMor} implies $\varphi\leq \psi$.

(2)
Every \CuMor{} is also a generalized \CuMor{}, and we therefore consider $\CatCuMor(S,T)$ as a subset of $\CatCuMor[S,T]$.
For $\varphi\in\CatCuMor[S,T]$, we have $\varphi\prec\varphi$ if and only if $\varphi$ is a $\CatCu$-morphism.
\end{rmks}

It follows from \autoref{prp:Q:genMorInQ} that $\prec$ is an auxiliary relation on $\CatCuMor[S,T]$ and that $(\CatCuMor[S,T],\prec)$ is a $\CatQ$-semigroup.
We may therefore apply the $\tau$-construction.

\begin{dfn}
\label{dfn:bivarCu:ihom}
Let $S$ and $T$ be \CuSgp{s}.
We define the \emph{internal hom} from $S$ to $T$ as the \CuSgp{}
\[
\ihom{S,T} :=\tau\big( \CatCuMor[S,T],\prec \big).
\]
We call $\ihom{S,T}$ the \emph{bivariant \CuSgp{}}, or the \emph{abstract bivariant Cuntz semigroup} of $S$ and $T$.
\end{dfn}

\begin{rmk}
\label{rmk:bivarCu:ihom}
Recall that a path in $\CatCu[S,T]$ is a map $f\colon I_\QQ\to\CatCu[S,T]$ such that $f(\lambda')\prec f(\lambda)$ whenever $\lambda',\lambda\in I_\QQ$ satisfy $\lambda'<\lambda$.
We often denote $f(\lambda)$ by $f_\lambda$ and we denote the path by $\pathCu{f}=(f_\lambda)_\lambda$.
By definition then, the elements of $\ihom{S,T}$ are equivalence classes of paths in the $\CatQ$-semigroup $(\CatCu[S,T],\prec)$.
\end{rmk}

\begin{pgr}
\label{pgr:bivarCu:ihom_functor}
We now show that the internal-hom in $\CatCu$ is functorial in both variables:
contravariant in the first and covariant in the second variable.

Let $T$ be a \CuSgp{}.
Considering $T$ as a $\CatQ$-semigroup, we have a contravariant functor $\CatQMor[\freeVar,T]\colon\CatQ\to\CatQ$ as in \autoref{pgr:Q:ihom_functor}.
Precomposing with the inclusion $\iota\colon\CatCu\to\CatQ$ from \autoref{pgr:Q:inclusionCuQ} and postcomposing with the functor $\tau\colon\CatQ\to\CatCu$, we obtain a functor $\ihom{\freeVar,T}\colon\CatCu\to\CatCu$.

Given \CuSgp{s} $S_1$ and $S_2$, and a \CuMor{} $\alpha\colon S_1\to S_2$, we use $\alpha^*$ to denote the induced \CuMor{} $\ihom{S_2,T}\to\ihom{S_1,T}$.
Thus, if we consider $\alpha$ as a $\CatQ$-morphism and if we let $\alpha^*_{\CatQ}\colon\CatQMor[S_2,T]\to\CatQMor[S_1,T]$ denote the induced $\CatQ$-morphism from \autoref{prp:Q:functorial}, then $\alpha^*$ is given as $\alpha^* := \tau( \alpha^*_{\CatQ} )$.

Analogously, given a \CuSgp{} $S$, we define the functor $\ihom{S,\freeVar}\colon\CatCu\to\CatCu$ as the composition of the functors $\iota$, the functor $\CatQMor[S,\freeVar]$ from \autoref{pgr:Q:ihom_functor} and $\tau$.

Given \CuSgp{s} $T_1$ and $T_2$, and a \CuMor{} $\beta\colon T_1\to T_2$, we use $\beta_*$ to denote the induced \CuMor{} $\ihom{S,T_1}\to\ihom{S,T_2}$.
If we consider $\beta$ as a $\CatQ$-morphism and if we let $\beta_*^{\CatQ}\colon\CatQMor[S,T_1]\to\CatQMor[S,T_2]$ denote the induced $\CatQ$-morphism from \autoref{prp:Q:functorial}, then $\beta_*$ is given as $\beta_* := \tau( \beta_*^{\CatQ} )$.

Thus, the internal-hom in the category $\CatCu$ is a bifunctor
\[
\ihom{\freeVar,\freeVar}\colon\CatCu\times\CatCu\to\CatCu.
\]
\end{pgr}

Next, we transfer the concept of the endpoint map from \autoref{dfn:Q:endpointMap} to the setting of bivariant \CuSgp{s}.
To simplify notation, we write $\sigma_{S,T}$ for $\varphi_{\Cu[S,T]}$, the endpoint map associated to the $\CatQ$-semigroup $\CatCuMor[S,T]$.
The next definition makes this precise.

\begin{dfn}
\label{dfn:bivarCu:endpointMap}
Let $S$ and $T$ be \CuSgp{s}.
We let $\sigma_{S,T}\colon\ihom{S,T}\to\CatCuMor[S,T]$ be defined by
\[
\sigma_{S,T}([\pathCu{f}])(a)
= \sup_{\lambda\in I_\QQ} f_\lambda(a),
\]
for a path $\pathCu{f}=(f_\lambda)_\lambda$ in $\CatCu[S,T]$ and $a\in S$.
We refer to $\sigma_{S,T}$ as the \emph{endpoint map}.
\end{dfn}

\begin{lma}
\label{prp:bivarCu:inducedBimor}
Let $S$, $T$ and $P$ be \CuSgp{s}, and let $\alpha\colon S\to\ihom{T,P}$ be a \CuMor{}.
Let $\sigma_{T,P}\colon\ihom{T,P}\to\CatCu[T,P]$ be the endpoint map from \autoref{dfn:bivarCu:endpointMap}.
Define $\bar{\alpha}\colon S\times T\to P$ by
\[
\bar{\alpha}(a,b) =\sigma_{T,P}(\alpha(a))(b),
\]
for $a\in S$ and $b\in T$.
Then $\bar{\alpha}$ is a $\CatCu$-bimorphism.
\end{lma}
\begin{proof}
We write $\sigma$ for $\sigma_{T,P}$.
To show that $\bar{\alpha}$ is a generalized \CuMor{} in the first variable, let $b\in T$.
Since $\alpha$ and $\sigma$ are both additive and order preserving, we conclude that $\bar{\alpha}(\freeVar,b)=\sigma(\alpha(\freeVar))(b)$ is additive and order preserving as well.
To show that $\bar{\alpha}(\freeVar,b)$ preserves suprema of increasing sequences, let $(a_n)_n$ be an increasing sequence in $S$.
Set $a:=\sup_n a_n$.
Since both $\alpha$ and $\sigma$ preserve suprema of increasing sequences, we obtain that
\[
\sigma(\alpha(a)) =\sup_n \sigma(\alpha(a_n)),
\]
in $\CatCu[T,P]$.
Since the supremum of an increasing sequence in $\CatCu[T,P]$ is the pointwise supremum, we get that $\bar{\alpha}(a,b)=\sup_n\bar{\alpha}(a_n,b)$, as desired.

For each $a\in S$, we have $\bar{\alpha}(a,\freeVar)=\sigma(\alpha(a))$, which is an element in $\CatCu[T,P]$.
Therefore, $\bar{\alpha}$ is a generalized \CuMor{} in the second variable.

Lastly, to show that $\bar{\alpha}$ preserves the joint way-below relation, let $a',a\in S$ and $b',b\in T$ satisfy $a'\ll a$ and $b'\ll b$.
Since $\alpha$ is a \CuMor{} we have $\alpha(a')\ll\alpha(a)$.
Using that $\sigma$ is a $\CatQ$-morphism, it follows that $\sigma(\alpha(a'))\prec \sigma(\alpha(a))$.
Therefore, applying the definition of the auxiliary relation $\prec$ at the second step, we obtain that
\[
\bar{\alpha}(a',b')
= \sigma(\alpha(a'))(b')
\ll \sigma(\alpha(a))(b)
= \bar{\alpha}(a,b),
\]
as desired.
\end{proof}

We omit the straightforward proof of the following result.

\begin{lma}
\label{prp:bivarCu:fullSubcatBimor}
Let $S,T$ and $P$ be \CuSgp{s}, and let $\varphi\colon S\times T\to P$ be a map.
Then $\varphi$ is a $\CatCu$-bimorphism if and only if $\varphi$, considered as a map between $\CatQ$-semigroups, is a $\CatQ$-bimorphism.
Thus, we have a natural bijection
\[
\CatQBimor(S\times T,P)
\cong \CatCuBimor(S\times T,P),
\]
which, moreover, respects the structure of the bimorphism sets as \pom{s}.
\end{lma}

\begin{lma}
\label{prp:bivarCu:bij_morToIhom_bimor}
Let $S, T$ and $P$ be \CuSgp{s}.
Then the assignment that sends a \CuMor{} $\alpha\colon S \to \ihom{T,P}$ to the $\CatCu$-bimorphism $\tilde{\alpha}\colon S\times T\to P$ given in \autoref{prp:bivarCu:inducedBimor} defines a natural bijection
\[
\CatCuMor\big( S, \ihom{T,P} \big)
\cong \CatCuBimor\big( S\times T, P \big),
\]
which respects the structure of the (bi)morphism sets as \pom{s}.
\end{lma}
\begin{proof}
By definition, we have $\CatCuMor\big( S, \ihom{T,P} \big) = \CatCuMor\big( S, \tau(\CatQMor[T,P]) \big)$.
Further, we have natural bijections, respecting the structure as \pom{s}, using \autoref{thm:Q:coreflection} at the first step, using \autoref{prp:Q:ihom} at the second step, and using \autoref{prp:bivarCu:fullSubcatBimor} at the last step:
\[
\CatCuMor\big( S, \tau(\CatQMor[T,P]) \big)
\cong \CatQMor\big( S, \CatQMor[T,P] \big)
\cong \CatQBimor\big( S\times T, P \big)
\cong \CatCuBimor\big( S\times T, P \big).
\]
It is straightforward to check that the composition of these bijections identifies a \CuMor{} $\alpha$ with the $\CatCu$-bimorphism $\tilde{\alpha}$ as defined in \autoref{prp:bivarCu:inducedBimor}.
\end{proof}

\begin{thm}
\label{prp:bivarCu:ihomAdjoint}
Let $S, T$ and $P$ be \CuSgp{s}.
Then there are natural bijections
\[
\CatCuMor\big( S, \ihom{T,P} \big)
\cong \CatCuBimor\big( S\times T, P \big)
\cong \CatCuMor\big( S\otimes T, P \big),
\]
which respect the structure of the (bi)morphism sets as \pom{s}.

The first bijection is given by assigning to a \CuMor{} $\alpha\colon S\to\ihom{T,P}$ the $\CatCu$-bimorphism $\tilde{\alpha}\colon S\times T\to P$ as in \autoref{prp:bivarCu:inducedBimor}, that is, $\tilde{\alpha}(a,b)=\sigma_{T,P}(\alpha(a))(b)$, for $(a,b)\in S\times T$.
The second bijection is given by assigning to a \CuMor{} $\beta\colon S\otimes T\to P$ the $\CatCu$-bimorphism $S\times T\to P$, $(a,b)\mapsto\beta(a\otimes b)$, for $(a,b)\in S\times T$.
\end{thm}
\begin{proof}
The first bijection is obtained from \autoref{prp:bivarCu:bij_morToIhom_bimor}.
The second bijection follows from \autoref{prp:prelim:tensCu}.
It is also shown in these results that the bijections respect the structure of the (bi)morphism sets as \pom{s}.
\end{proof}

Let $T$ be a \CuSgp{}.
We consider the functor $\freeVar\otimes T\colon\CatCu\to\CatCu$ given by tensoring with $T$.
It follows from \autoref{prp:bivarCu:ihomAdjoint} that the functor $\ihom{T,\freeVar}$ is a right adjoint of $\freeVar\otimes T$.
By definition, this shows that the monoidal category $\CatCu$ is closed, and we obtain the following result:

\begin{thm}
\label{prp:bivarCu:CuClosed}
The category $\CatCu$ of abstract Cuntz semigroups is a closed, symmetric, mo\-noi\-dal category.
\end{thm}

Every closed symmetric monoidal category is enriched over itself, as noted in \autoref{pgr:prelim:closedCat}.
Given \CuSgp{s} $S$ and $T$, the \CuSgp{} $\ihom{S,T}$ plays the role of morphisms from $S$ to $T$.
First, we show that \CuMor{s} $S\to T$ correspond to compact elements in $\ihom{S,T}$.

\begin{prp}
\label{prp:bivarCu:cpctElt}
Let $S$ and $T$ be \CuSgp{s}.
Then there is a natural bijection $\ihom{S,T}_c\cong\CatCuMor(S,T)$, between \CuMor{s} $S\to T$ and compact elements in $\ihom{S,T}$.
A \CuMor{} $\varphi\colon S\to T$ is associated with the class in $\ihom{S,T}$ of the constant path with value $\varphi$.
Conversely, given a compact element in $\ihom{S,T}$ represented by a path $(\varphi_\lambda)_\lambda$, then for $\lambda$ close enough to $1$ the map $\varphi_\lambda$ is a \CuMor{} and independent of $\lambda$. 
\end{prp}
\begin{proof}
It is straightforward to verify that the described associations are well-defined and inverses of each other.
Alternatively, note that for every \CuSgp{} $P$, there is a natural identification of $P_c$ with $\CatCu(\NNbar,P)$, by associating to a \CuMor{} $\varphi\colon\NNbar\to P$ the compact element $\varphi(1)$.
Using this fact at the first step, using \autoref{prp:bivarCu:ihomAdjoint} at the second step, and using the isomorphism $\NNbar\otimes S\cong S$ at the third step, we obtain that
\[
\ihom{S,T}_c
\cong \CatCuMor\big( \NNbar,\ihom{S,T} \big)
\cong \CatCuMor\big( \NNbar\otimes S,T \big)
\cong \CatCuMor(S,T),
\]
as desired.
\end{proof}

In particular, the identity \CuMor{} $\id_S\colon S\to S$ naturally corresponds to a compact element in $\ihom{S,S}$, also denoted by $\id_S$.
Further, $\id_S$ also naturally corresponds to a \CuMor{} $j_S\colon\NNbar\to\ihom{S,S}$, which is the identity of $S$ for the enrichment of $\CatCu$ over itself.

Given \CuSgp{s} $S$ and $T$, recall that the \emph{counit map}, or \emph{evaluation map} is the \CuMor{} $\counit_T^S \colon\ihom{S,T}\otimes S\to T$  that corresponds to $\id_{\ihom{S,T}}$ under the identification $\CatCuMor \big( \ihom{S,T},\ihom{S,T}\big) \cong \CatCuMor\big( \ihom{S,T}\otimes S,T \big)$.

Given \CuSgp{s} $S$, $T$ and $P$, consider the following \CuMor{}:
\[
(\ihom{T,P}\otimes\ihom{S,T})\otimes S
\xrightarrow{\cong} \ihom{T,P}\otimes(\ihom{S,T}\otimes S)
\xrightarrow{\id\otimes\counit_T^S} \ihom{T,P}\otimes T
\xrightarrow{\counit_P^T} P.
\]
Under the identification $\CatCuMor\big( \ihom{T,P} \otimes \ihom{S,T}, \ihom{S,P} \big)
\cong \CatCuMor\big( \ihom{T,P} \otimes \ihom{S,T} \otimes S, P \big)$,
the above \CuMor{} corresponds to a \CuMor{}
\[
\circ\colon \ihom{T,P} \otimes \ihom{S,T} \to \ihom{S,P},
\]
that we will call the \emph{composition product}.
The composition product implements the composition of morphisms when viewing the category $\CatCu$ as enriched over itself. (See \cite{AntPerThi17pre:AbsbivarII} for further details.) 

\begin{rmk}
The order of the product in $KK$-theory is reversed from the one used here for the category $\CatCu$, that is, given \ca{s} $A,B$ and $D$, the product in $KK$-theory is as a bilinear map
\[
KK(A,D)\times KK(D,B) \to KK(A,B);
\]
see \cite[Section~18.1, p166]{Bla98KThy} and \cite[Before Lemma~2.2.9, p.73]{JenTho91ElementsKK}.

We have mainly two reasons for our choice of ordering for the composition product in the category $\CatCu$:
First, the composition product extends the usual composition of \CuMor{s} and our choice is compatible with the standard notation for composition of maps.
Second, our ordering agrees with that of the composition law of internal-homs in closed categories;
see \cite[Section~1.6, p.15]{Kel05EnrichedCat}.
\end{rmk}

\subsection{Examples}
\label{sec:bivarCu:examples}

In this subsection, we compute several examples of bivariant \CuSgp{s} $\ihom{S,T}$.
We mostly consider the case that $S$ and $T$ are the Cuntz semigroups of the Jacelon-Razak algebra $\mathcal{W}$, of the Jiang-Su algebra $\mathcal{Z}$, of a UHF-algebra of infinite type, or of the Cuntz algebra $\mathcal{O}_2$.


Recall that $\PPbar$ denotes the semigroup $[0,\infty]$ with the usual order and addition. It is known that $\PPbar\cong\Cu(\mathcal{W})$, the Cuntz semigroup of the Jacelon-Razak algebra $\mathcal{W}$ introduced in \cite{Jac13Projectionless} (see~\cite{Rob13Cone}). The product of real numbers extends to a natural product on $[0,\infty]$ giving $\PPbar$ the structure of a solid $\CatCu$-semiring;
see \cite[Definition~7.1.5, Example~7.1.7]{AntPerThi14arX:TensorProdCu}.

Let $M_1$ be defined as in \autoref{exa:Q:M}.
By \autoref{prp:Q:M}, $M_1$ is the Cuntz semigroup of a $\mathrm{II}_1$-factor $M$.

\begin{prp}
\label{prp:bivarCu:ihomPP}
There is a natural isomorphism $\ihom{\PPbar,\PPbar} \cong M_1$.
\end{prp}
\begin{proof}
We show that the $\CatQ$-semigroup $(\CatCuMor[\PPbar,\PPbar],\prec)$ is isomorphic to $(\PPbar,\prec_1)$, where $\prec_1$ is the auxiliary relation defined in \autoref{exa:Q:M}.
Applying the $\tau$-construction, and using the arguments in \autoref{exa:Q:M} at the last step, we then obtain
\[
\ihom{\PPbar,\PPbar}
= \tau\left( \CatCuMor[\PPbar,\PPbar],\prec \right)
\cong \tau\left( \PPbar,\prec_1 \right) = M_1.
\]

Since $\PPbar$ is a solid $\CatCu$-semiring, any generalized \CuMor{} $\varphi\colon\PPbar\to\PPbar$ is $\PPbar$-linear (see \cite[Proposition~7.1.6]{AntPerThi14arX:TensorProdCu}). Thus, we have $\varphi(x)=\varphi(1)x$ for all $x\in\PPbar$.
We may identify $\CatCuMor[\PPbar,\PPbar]$ with $\PPbar$ by $\varphi\mapsto\varphi(1)$ and this is easily seen to be an additive order-isomorphism.
To conclude the argument, we need to show that under this identification, the auxiliary relation $\prec$ on $\CatCuMor[\PPbar,\PPbar]$ corresponds to the auxiliary relation $\prec_1$ on $\PPbar$ as defined in \autoref{exa:Q:M}.

Let $\varphi,\psi\in\CatCuMor[\PPbar,\PPbar]$.
Clearly $\varphi\prec\psi$ implies $\varphi(1)\leq\psi(1)$.
Moreover, if $\varphi(1)=\psi(1)=\infty$, then $\varphi\nprec\psi$, since $\varphi(1)=\infty\nll\infty=\psi(2)$ while $1\ll 2$.
Thus, $\varphi\prec\psi$ implies $\varphi(1)\prec_1\psi(1)$.
Conversely, assume that $\varphi(1)\prec_1\psi(1)$.
By definition, $\varphi(1)$ is finite, and $\varphi(1)\leq\psi(1)$.
To show that $\varphi\prec\psi$, let $s,t\in\PPbar$ satisfy $s\ll t$.
Using that $\varphi(1)$ is finite at the second step, we deduce that
\[
\varphi(s) = \varphi(1) s \ll \varphi(1) t \leq \psi(1)t = \psi(t). \qedhere
\]
\end{proof}

We let $Z$ be the disjoint union $\NN\sqcup(0,\infty]$, with elements in $\NN$ being compact, and with elements in $(0,\infty)$ being soft. It is known that $Z$ is isomorphic to the Cuntz semigroup of the Jiang-Su algebra $\mathcal{Z}$ introduced in \cite{JiaSu99Projectionless} (see \cite{PerTom07Recasting} and also \cite{BroTom07ThreeAppl}). To distinguish elements in both parts, we write $a'$ (with a prime symbol) for the soft element of value $a$.
For example, the compact one, denoted $1$, corresponds the class of the unit in $\mathcal{Z}$;
and the soft one, denoted $1'$, corresponds to the class of a positive element $x$ in $\mathcal{Z}$ that has spectrum $[0,1]$ and with $\lim_{n\to\infty}\tau(x^{1/n})=1$, for the unique trace $\tau$ on $\mathcal{Z}$.

Order and addition are the usual inside the components $\NN$ and $(0,\infty]$ of $Z$.
Given $a\in\NN$ and $b'\in(0,\infty]$, we have $a+b'=(a+b)'$ (the soft part is absorbing), and we have $a\leq b'$ if and only $a'<b'$, and we have $a\geq b'$ if and only if $a'\geq b'$.

We have a natural commutative product in $Z$, extending the natural products in the components $\NN$ and $(0,\infty]$, and such that $0a=0$ for every $a\in Z$, and such that $ab'=(ab)'$ for $a\in\NN_{>0}$ and $b'\in(0,\infty)$.
Note that $1$ (the compact one) is a unit for this semiring, but $1'$ is not.
Indeed, we have $1'1=1'$. This gives $Z$ the structure of a solid $\CatCu$-semiring
T
(see \cite[Section~7.3]{AntPerThi14arX:TensorProdCu}).

Given a supernatural number $q$ satisfying $q=q^2\neq 1$, we let $\NN[\tfrac{1}{q}]$ denote the set of nonnegative rational numbers whose denominators divide $q$, with usual addition.
Let $R_q=\NN[\tfrac{1}{q}]\sqcup(0,\infty]$, with elements in $\NN[\tfrac{1}{q}]$ being compact, and with elements in $(0,\infty]$ being soft.
Addition and order in $R_q$ is defined in analogy with $Z$.
If $M_q$ denotes the UHF-algebra of type $q$, then it is known that $\Cu(M_q)\cong R_q$. Analogous to the case for $Z$, we can define a multiplication on $R_q$, giving it the structure of a solid $\CatCu$-semiring
(see \cite[Section~7.4]{AntPerThi14arX:TensorProdCu}).

We exclude zero as a supernatural number.
However, $1$ is supernatural number that agrees with its square.
It is consistent to let $R_1$ denote the Cuntz semigroup of the Jiang-Su algebra $\mathcal{Z}$.
Thus, we set $R_1:=Z$, which simplifies the statement of \autoref{prp:bivarCu:ihomRpRq} below.

Given supernatural numbers $p$ and $q$ satisfying $p=p^2$ and $q=q^2$, we have $R_p\otimes R_q\cong R_{pq}$.
In particular, $Z\otimes R_p = R_1\otimes R_p \cong R_p$.
Moreover, if we let $Q=\QQ^+\sqcup (0,\infty]$, then $Q$ is isomorphic to the Cuntz semigroup of the universal UHF-algebra (whose $K_0$-group is isomorphic to the rational numbers). We have $Q\otimes R_p\cong Q$.

\begin{prp}
\label{prp:bivarCu:ihomRpRq}
Let $p$ and $q$ be supernatural numbers with $p=p^2$ and $q=q^2$.
If $p$ divides $q$, then $\ihom{R_p,R_q}\cong R_q$.
If $p$ does not divide $q$, then $\CatCu(R_p,R_q)=\{0\}$ and $\ihom{R_p,R_q}\cong\PPbar$.
\end{prp}
\begin{proof}
First, assume that $p$ divides $q$.
Then $R_p\cong R_p\otimes R_p$ and $R_q\cong R_q\otimes R_p$.
Let $\varphi\colon R_p\to R_q$ be a generalized \CuMor.
It follows from \cite[Proposition~7.1.6]{AntPerThi14arX:TensorProdCu} that $\varphi$ is $R_p$-linear.
Thus, $\varphi$ is determined by the image of the unit.
Moreover, for every $a\in R_q$, there is a generalized \CuMor{} $\varphi\colon R_p\to R_q$ with $\varphi(1)=a$, given by $\varphi(t)=ta$ for $t\in R_p$.
Thus, there is a bijection $\CatCuMor[R_p,R_q]\cong R_q$ given by identifying $\varphi$ with $\varphi(1)$.
It is straightforward to check that under this identification, the relation $\prec$ on $\CatCuMor[R_p,R_q]$ corresponds precisely to the way-below relation on $R_q$.
It follows that
\[
\ihom{R_p,R_q} = \tau\big( \CatCuMor[R_p,R_q],\prec \big) \cong \tau\big( R_q,\ll \big) \cong R_q,
\]
as desired.

Assume now that $p$ does not divide $q$.
Let $r$ be a prime number dividing $p$ but not $q$.
Every element of $R_p$ is divisible by arbitrary powers of $r$.

On the other hand, we claim that only the soft elements of $R_q$ are divisible by arbitrary powers of $r$.
Indeed, every element of $R_q$ is either compact or nonzero and soft.
Moreover, the sum of a nonzero soft element with any other element in $R_q$ is soft.
It follows that if a compact element of $R_q$ is divisible in $R_q$ then it is also divisible in the monoid of compact elements of $R_q$, which we identify with $\NN[\tfrac{1}{q}]$.
However, since $r$ does not divide $q$, the only element in $\NN[\tfrac{1}{q}]$ that is divisible by arbitrary powers of $r$ is the zero element, which is soft.

It follows that every generalized \CuMor{} $R_p\to R_q$ has its image contained in the soft part of $R_q$.
In particular, if $\varphi\colon R_p\to R_q$ is a $\CatCu$-morphism, then every compact element of $R_p$ is sent to zero by $\varphi$.
Using that $R_p$ is simple, it follows that $\varphi$ is the zero map.
Thus, $\CatCu(R_p,R_q)=\{0\}$, as desired.

We identify $\PPbar$ with the soft part of $R_p$, and similarly for $R_q$.
Let $\varphi\colon R_p\to R_q$ be a generalized \CuMor{}.
We have seen that $\varphi(1)$ belongs to $\PPbar=(R_q)_\txtSoft$.
Moreover, for every $a\in\PPbar$ there is a generalized \CuMor{} $\varphi\colon R_p\to \PPbar\subseteq R_q$ with $\varphi(1)=a$, given by $\varphi(t)=ta$ for $t\in R_p$.
Thus, there is a bijection $\CatCuMor[R_p,R_q]\cong \PPbar$ given by identifying $\varphi$ with $\varphi(1)$.
It is straightforward to check that under this identification, the relation $\prec$ on $\CatCuMor[R_p,R_q]$ corresponds to the way-below relation on $\PPbar$.
As above, it follows that
\[
\ihom{R_p,R_q} = \tau\big( \CatCuMor[R_p,R_q],\prec \big) \cong \tau\big( \PPbar,\ll \big) \cong \PPbar. \qedhere
\]
\end{proof}

\begin{exa}
\label{exa:bivarCu:ihomRpRq}
By \autoref{prp:bivarCu:ihomRpRq}, there are natural isomorphisms $\ihom{Z,Z} \cong Z$ and $\ihom{Q,Q} \cong Q$.
More generally, for every supernatural number $q$ with $q=q^2$, there are natural isomorphisms $\ihom{Z,R_q} \cong R_q$ and $\ihom{R_q,Q} \cong Q$.
\end{exa}

\begin{exa}
\label{exa:bivarCu:ihomRpP}
Let $q$ be a supernatural number with $q=q^2$.
Then there are natural isomorphisms $\ihom{R_q,\PPbar}\cong\PPbar$ and $\ihom{\PPbar,R_q}\cong M_1$, which can be proved similarly as Propositions~\ref{prp:bivarCu:ihomRpRq} and~\ref{prp:bivarCu:ihomPP}.
In particular, we have $\ihom{Z,\PPbar}\cong\PPbar$ and $\ihom{\PPbar,Z}\cong M_1$.
\end{exa}

Given $k\in\NN$, we set $E_k:=\{0,1,\dots,k,\infty\}$, equipped with the natural order and addition as a subset of $\NNbar$, with the convention that $a+b=\infty$ whenever $a+b>k$ in $\NNbar$.
With the obvious multiplication, $E_k$ is a solid $\CatCu$-semiring (see, e.g. \cite[Example~8.1.2]{AntPerThi14arX:TensorProdCu}).
Note that $E_0=\{0,\infty\}$ is the Cuntz semigroup of the Cuntz algebra $\mathcal{O}_2$ (or of any other simple, purely infinite \ca{}).

\begin{prp}
\label{prp:bivarCu:ihomEkEl}
Let $k,l$ be natural numbers.
Let $\lceil{\frac{l+1}{k+1}}\rceil$ denote the smallest natural number larger than or equal to ${\frac{l+1}{k+1}}$.
Then $\ihom{E_k,E_l}$ is isomorphic to the sub-$\CatCu$-semigroup $\{0,\lceil{\frac{l+1}{k+1}}\rceil,\dots,l,\infty\}$ of $E_l$.
\end{prp}
\begin{proof}
Let $\varphi\colon E_k\to E_l$ be a generalized $\CatCu$-morphism.
Then $\varphi$ is determined by the image of $1$, which can be zero or any element $a\in E_l$ such that $(k+1)a=\infty$.
Thus, for every $a\geq \frac{l+1}{k+1}$ there is a unique generalized $\CatCu$-morphism $E_k\to E_l$ given by $x\mapsto ax$.
Moreover, each such a map preserves the way-below relation and is therefore a $\CatCu$-morphism.
The desired result follows.
\end{proof}

\begin{exa}
\label{exa:bivarCu:ihomEkEl}
There is a natural isomorphism $\ihom{\{0,\infty\},\{0,\infty\}}\cong\{0,\infty\}$, and more generally $\ihom{E_k,E_k}\cong E_k$ for every $k\in\NN$.
\end{exa}

\section{Applications to \texorpdfstring{\ca{s}}{C*-algebras}}
\label{sec:appl}

Given \ca{s} $A$ and $B$, a map $\varphi\colon A\to B$ is called \emph{completely positive contractive} (abbreviated c.p.c.) if it is linear, contractive and for each $n\in\NN$ the amplification to $n\times n$-matrices $\varphi\otimes\id\colon A\otimes M_n \to B\otimes M_n$ is positive.
Every c.p.c.\ map $\varphi\colon A\to B$ induces a contractive, positive map $\varphi\otimes\id\colon A\otimes\KK\to B\otimes\KK$.

Two elements $a$ and $b$ in a \ca{} are called \emph{orthogonal}, denoted $a\perp b$, if $ab=a^*b=ab^*=a^*b^*=0$.
If $a$ and $b$ are self-adjoint, then $a\perp b$ if and only if $ab=0$.
A c.p.c.\ map $\varphi$ is said to have \emph{order-zero} if for all $a,b\in A$ we have that $a\perp b$ implies $\varphi(a)\perp\varphi(b)$.
We denote the set of c.p.c.\ order-zero maps by $\cpc{A,B}$.

The concept of c.p.c.\ order-zero maps was studied by Winter and Zacharias, \cite{WinZac09CpOrd0}, who also gave a useful structure theorem for such maps.
We present their result in a slightly different way.

\begin{thm}[{Winter and Zacharias, \cite[Theorem~3.3]{WinZac09CpOrd0}}]
\label{prp:appl:oz:charWZ}
Let $A$ and $B$ be \ca{s}, and let $\varphi\colon A\to B$ be a c.p.c.\ order-zero map.
Set $C:=C^*(\varphi(A))$, the sub-\ca{} of $B$ generated by the image of $\varphi$.
Then there exists a unital \stHom{} $\pi_\varphi\colon \widetilde{A}\to M(C)$, from the minimal unitization of $A$ to the multiplier algebra of $C$,  such that
\[
\varphi(ab) =\varphi(a)\pi_\varphi(b) = \pi_\varphi(a)\varphi(b),
\]
for $a,b\in\widetilde{A}$.
In particular, the element $h:=\varphi(1_{\widetilde{A}})$ is contractive, positive, it commutes with the image of $\pi_\varphi$, and $\varphi(a)=h\pi_\varphi(a)=\pi_\varphi(a)h$ for all $a\in A$.
\end{thm}

This structure theorem has many interesting applications.
For instance, it implies that c.p.c.\ order-zero maps induce generalized $\CatCu$-morphisms.
Let us recall some details.
Let $\varphi\colon A\to B$ be a c.p.c.\ order-zero map.
Then the amplification $\varphi\otimes\id\colon A\otimes\KK\to B\otimes\KK$ is a c.p.c.\ order-zero map as well;
see \cite[Corollary~4.3]{WinZac09CpOrd0}.
Define $\Cu[\varphi]\colon\Cu(A)\to\Cu(B)$ by
\[
\Cu[\varphi]([a]) := [(\varphi\otimes\id)(a)],
\]
for $a\in (A\otimes\KK)_+$.
Then $\Cu[\varphi]$ is a generalized \CuMor;
see \cite[Corollary~4.5]{WinZac09CpOrd0} and \cite[2.2.7, 3.2.5]{AntPerThi14arX:TensorProdCu}.
We thus obtain a natural map
\[
\cpc{A,B} \to \CatCuMor[\Cu(A),\Cu(B)].
\]
Below, we will show that this map factors through $\ihom{\Cu(A),\Cu(B)}$.

\begin{pgr}
The theorem of Winter and Zacharias also allows us to define functional calculus for order-zero maps:
Let $\varphi\colon A\to B$ be a c.p.c.\ order-zero map.
Choose $C$, $\pi_\varphi$ and $h$ as in \autoref{prp:appl:oz:charWZ}.
Given a continuous function $f\colon[0,1]\to[0,1]$ with $f(0)=0$, we define $f(\varphi)\colon A\to B$ by $f(\varphi)(a) := f(h)\pi_\varphi(a)$ for $a\in A$;
see \cite[Corollary~4.2]{WinZac09CpOrd0}.

In particular, this allows us to define `cut-downs' of c.p.c.\ order-zero maps:
Given $\varepsilon>0$, we may apply the function $(\freeVar-\varepsilon)_+$ to $\varphi$.
To simplify notation, we set $\varphi_\varepsilon:=(\varphi-\varepsilon)_+$.
Thus, for $a\in A$ we have $\varphi_\varepsilon(a) = (h-\varepsilon)_+\pi_\varphi(a)$.
\end{pgr}

\begin{thm}
\label{prp:appl:oz:oz_ind_ihom}
Let $A$ and $B$ be \ca{s}, and let $\varphi\colon A\to B$ be a c.p.c.\ order-zero map.
For each $\varepsilon>0$, let $f_\varepsilon\colon\Cu(A)\to\Cu(B)$ be the generalized $\CatCu$-morphism induced by the c.p.c.\ order-zero map $\varphi_\varepsilon\colon A\to B$.
Then $\pathCu{f}=(f_{1-\lambda})_\lambda$ is a path in $\CatCuMor[\Cu(A),\Cu(B)]$.
Moreover, the endpoint of $\pathCu{f}$ is $\CatCu[\varphi]$, the generalized $\CatCu$-morphism induced by $\varphi$.
\end{thm}
\begin{proof}
We have already observed that every $f_\varepsilon$ is a generalized $\CatCu$-morphism.
To verify that $(f_{1-\lambda})_\lambda$ is a path, we need to show that $f_{\varepsilon'}\prec f_\varepsilon$ for $\varepsilon'>\varepsilon>0$.
Since $f_{\varepsilon+\delta}=(f_{\varepsilon})_{\delta}$, it is enough to show the following:

Claim:
We have $f_\varepsilon\prec f$.
To show the claim, let $a,b\in(A\otimes\KK)_+$ such that $[a]\ll[b]$ in $\Cu(A)$.
Recall that two positive elements $x$ and $y$ in a \ca{} satisfy $[x]\ll[y]$ if and only if there exists $\delta>0$ with $[x]\leq[(y-\delta)_+]$.
Thus, we can choose $\delta>0$ such that $[a]\leq[(b-\delta)_+]$.
Note that if $x$ and $y$ are commuting positive elements in a \ca{}, then $(x-\varepsilon)_+(y-\delta)_+\leq (xy-\varepsilon\delta)_+$.
Using this at the last step, we deduce that
\begin{align*}
\varphi_\varepsilon(a)
\precsim \varphi_\varepsilon((b-\delta)_+)
&= (h-\varepsilon)_+\pi_\varphi((b-\delta)_+) \\
&= (h-\varepsilon)_+ (\pi_\varphi(b)-\delta)_+
\leq (h \pi_\varphi(b)-\varepsilon\delta)_+
= (\varphi(b)-\varepsilon\delta)_+,
\end{align*}
which implies that
\[
f_\varepsilon([a]) = [\varphi_\varepsilon(a)] \ll [\varphi(b)] = f([b]),
\]
as desired.
This proves the claim and shows that $\pathCu{f}$ is a path.

Let $f$ be the generalized $\CatCu$-morphism induced by $\varphi$.
To show that the endpoint of $\pathCu{f}$ is $f$, let $a\in(A\otimes\KK)_+$.
We have
\[
\lim_{\lambda\to 1} \varphi_{1-\lambda}(a)
= \lim_{\varepsilon\to 0} \varphi_\varepsilon(a)
= \lim_{\varepsilon\to 0} (h-\varepsilon)_+\pi_\varphi(a)
= h\pi_\varphi(a)
= \varphi(a).
\]
This implies that $\sup_{\lambda<1}f_\lambda([a])=f([a])$ in $\Cu(A)$, as desired.
\end{proof}

\begin{dfn}
\label{dfn:appl:oz:oz_ind_ihom}
Let $A$ and $B$ be \ca{s}, and let $\varphi\colon A\to B$ be a c.p.c.\ order-zero map.
We let $\Cu(\varphi)$ be the element in $\ihom{\Cu(A),\Cu(B)}$ that is the class of the path  $(\Cu[\varphi_{1-\lambda}])_\lambda$ as constructed in \autoref{prp:appl:oz:oz_ind_ihom}.
\end{dfn}

\begin{rmk}
\label{rmk:appl:oz:oz_ind_ihom}
Let $\varphi\colon A\to B$ be a \stHom.
In the definition of the functor $\Cu\colon\CatCa\to\CatCu$ we denoted $\Cu(\varphi)$ as the $\CatCu$-morphism $\Cu(A)\to\Cu(B)$ given by $\Cu(\varphi)([a])=[(\varphi\otimes\id)(a)]$ for $a\in (A\otimes\KK)_+$.

On the other hand, in \autoref{dfn:appl:oz:oz_ind_ihom} we defined $\Cu(\varphi)$ as the class of the path $(\Cu[\varphi_{1-\lambda}])_\lambda$ as constructed in \autoref{prp:appl:oz:oz_ind_ihom}.
Given $\varepsilon>0$, it is easy to verify that $\varphi_\varepsilon=(1-\varepsilon)_+\varphi$.
It follows that $\Cu[\varphi_\varepsilon]=\Cu[\varphi]$ for $\varepsilon\in[0,1)$.
Thus, the path $(\Cu[\varphi_{1-\lambda}])_\lambda$ is constant with value $\CatCu[\varphi]$.

We identify a $\CatCu$-morphism $f\colon\Cu(A)\to\Cu(B)$ with the compact element in $\ihom{\Cu(A),\Cu(B)}$ given by the constant path with value $f$;
see \autoref{prp:bivarCu:cpctElt}.
It follows that the notation $\CatCu(\varphi)$ for a \stHom{} $\varphi$ is unambiguous.
\end{rmk}

\begin{pgr}
\label{pgr:appl:cpcToIhom}
The functor $\CatCa\to\CatCu$ defines a map
\[
\CatCu\colon\Hom(A,B) \to \CatCu(\Cu(A),\CatCu(B)).
\]
By \autoref{dfn:appl:oz:oz_ind_ihom} we obtain a well-defined map
\[
\cpc{A,B} \to \ihom{\Cu(A),\Cu(B)}.
\]
As noticed in \autoref{rmk:appl:oz:oz_ind_ihom}, these assignemnts are compatible, which means that the following diagram commutes:
\[
\xymatrix{
\cpc{A,B} \ar[r]^-{\CatCu} & \ihom{\Cu(A),\Cu(B)} \\
\Hom(A,B) \ar[r]^-{\CatCu} \ar@{^{(}->}[u] & \CatCu(\Cu(A),\CatCu(B)) \ar@{^{(}->}[u] .
}
\]
\end{pgr}

\begin{pbm}
\label{pbm:appl:cpcToIhom}
Study the properties of the map $\cpc{A,B} \to \ihom{\Cu(A),\Cu(B)}$.
In particular, when is this map surjective?
\end{pbm}

\begin{exa}
\label{exa:appl:cpcWW}
Recall that $\mathcal{W}$ denotes the Jacelon-Razak algebra.
We know that $\Cu(\mathcal{W})\cong\PPbar$.
By \autoref{prp:bivarCu:ihomPP}, we have $\ihom{\PPbar,\PPbar}\cong M_1$, and recall that $M_1=[0,\infty)\sqcup (0,\infty]$.
We claim that the map
\[
\cpc{\mathcal{W},\mathcal{W}} \to \ihom{\Cu(\mathcal{W}),\Cu(\mathcal{W})} \cong \ihom{\PPbar,\PPbar}\cong M_1
\]
is surjective.

The idea is to choose a unital, simple, AF-algebra $A$ with unique tracial state and a suitable element $x\in (A\otimes\KK)_+$ and consider the map $\mathcal{W}\to\mathcal{W}\otimes A$, given by $y\mapsto y\otimes x$, followed by a ${}^*$-isomorphism $\mathcal{W}\otimes A\cong \mathcal{W}$.

Let $A$ be a unital, simple AF-algebra with unique tracial state.
We claim that $\mathcal{W}\otimes A\cong\mathcal{W}$.
By construction, $\mathcal{W}$ is an inductive limit of the building blocks considered by Razak in \cite{Raz02ClassifProjectionless}.
Since $A$ is an AF-algebra, $\mathcal{W}\otimes A$ is an inductive limit of Razak building blocks as well.
Since $A$ is simple and has a unique tracial state, $\mathcal{W}$ and $\mathcal{W}\otimes A$ have the same invariant used for the classification \cite[Theorem~1.1]{Raz02ClassifProjectionless}, which gives the desired ${}^*$-isomorphism $\mathcal{W}\otimes A\cong\mathcal{W}$.

Given $a\in M_1$, let us define a c.p.c.\ order-zero map $\mathcal{W}\to\mathcal{W}$ corresponding to $a$.
We distinguish two cases:

Case~1:
Assume that $a$ is nonzero and soft.
Let $\mathcal{U}$ denote the universal UHF-algebra.
We have $\Cu(\mathcal{U})\cong \QQ_+ \sqcup (0,\infty]$.
We consider $a$ as a soft element in $\Cu(\mathcal{U})_\txtSoft = [0,\infty]$.
Choose $x_a\in(\mathcal{U}\otimes\KK)_+$ with Cuntz class $a$.
(For example, let $x_a$ be a positive element with spectrum $[0,1]$ - ensuring that its Cuntz class is soft - and such that for the unique normalized extended trace $\tau\colon (\mathcal{U}\otimes\KK)_+\to[0,\infty]$ we have $\lim_{n\to\infty} \tau(x_a^{1/n}) = a$.)

Consider the map $\varphi_a\colon\mathcal{W}\to\mathcal{W}\otimes\mathcal{U}$ given by $\varphi_a(y)=y\otimes x_a$ for $y\in\mathcal{W}$.
It is easy to see that $\varphi_a$ is a c.p.c.\ order-zero map.
Let $\psi\colon \mathcal{W}\otimes\mathcal{U}\to\mathcal{W}$ be an isomorphism.
Then $\psi\circ\varphi_a$ is a c.p.c.\ order-zero map $\mathcal{W}\to\mathcal{W}$ with the desired properties.

Case 2: Assume that $a$ is compact.
We claim that there exists a unital, simple AF-algebra $A$ with unique normalized trace $\tau\colon(A\otimes\KK)_+\to[0,\infty]$ and a projection $p_a\in(A\otimes\KK)_+$ with $\tau(p_a)=a$.
Indeed, if $a$ is rational, then we can take $A=\mathcal{U}$.
If $a$ is irrational, then we use that $\ZZ+a\ZZ$ is a dimension group for the order and addition inherited as a subgroup of $\RR$.
Moreover, $\ZZ+a\ZZ$ has a unique normalized state.
It follows that there is a unique unital AF-algebra $A$ such that $(K_0(A),K_0(A)_+,[1])$ is isomorphic to $(\ZZ+a\ZZ,(\ZZ+a\ZZ)\cap[0,\infty),1)$.
By construction, there exists a projection $p_a\in A\otimes\KK $ with $\tau(p_a)=a$.

Define $\varphi_a\colon\mathcal{W}\to\mathcal{W}\otimes A$ by $\varphi_a(y)=y\otimes p_a$ for $y\in\mathcal{W}$.
Then $\varphi_a$ is a \stHom.
Postcomposing with a ${}^*$-isomorphism $\mathcal{W}\otimes R_\theta\cong\mathcal{W}$, we obtain a \stHom{} $\mathcal{W}\to\mathcal{W}$ with the desired properties.
\end{exa}

\begin{exa}
\label{exa:appl:cpcToIhom}
With similar methods as in \autoref{exa:appl:cpcWW}, one can show that the map $\cpc{A,B} \to \ihom{\Cu(A),\Cu(B)}$ is surjective whenever $A$ and $B$ are any of the following \ca{s}:
a UHF-algebra of infinite type, the Jiang-Su algebra, the Jacelon-Razak algebra $\mathcal{W}$.
\end{exa}

\begin{rmk}
In \cite[Definition~2.27]{BosTorZac16arX:BivarThyCu}, Bosa, Tornetta and Zacharias introduced a bivariant Cuntz semigroup, denoted $\mathop{WW}(A,B)$, as suitable equivalence classes of c.p.c.\ order-zero maps $A\otimes\KK\to B\otimes\KK$.
It would be interesting to study if the map from \autoref{pbm:appl:cpcToIhom} factors through $\mathop{WW}(A,B)$, that is, if the following diagram can be completed to be commutative:
\[
\xymatrix{
\cpc{A,B} \ar[r] \ar[d] & \ihom{\Cu(A),\Cu(B)} \\
\mathop{WW}(A,B) \ar@{-->}[ur]
}.
\]
Observe that, in order for this to be satisfied, one needs to show that, given $\varphi$ and $\psi$ in $\cpc{A,B}$ such that $\varphi\precsim\psi$ in the sense of \cite{BosTorZac16arX:BivarThyCu} then, for $\epsilon>0$, there is $\delta>0$ such that $\Cu[\varphi_{1-\epsilon}]\prec\Cu[\psi_{1-\delta}]$.
\end{rmk}



\providecommand{\etalchar}[1]{$^{#1}$}
\providecommand{\bysame}{\leavevmode\hbox to3em{\hrulefill}\thinspace}
\providecommand{\noopsort}[1]{}
\providecommand{\mr}[1]{\href{http://www.ams.org/mathscinet-geNovember 21, 2018.
		2010
		Mathematics Subject Classification.
		Primary 06B35, 06F05, 15A69, 46L05. Secondary
		06B30, 06F25, 13J25, 16W80, 16Y60, 18B35, 18D20, 19K14, 46L
		06, 46M15, 54F05.
		Key words and phrases.
		Cuntz semigroup, tensor product, continuous postitem?mr=#1}{MR~#1}}
\providecommand{\zbl}[1]{\href{http://www.zentralblatt-math.org/zmath/en/search/?q=an:#1}{Zbl~#1}}
\providecommand{\jfm}[1]{\href{http://www.emis.de/cgi-bin/JFM-item?#1}{JFM~#1}}
\providecommand{\arxiv}[1]{\href{http://www.arxiv.org/abs/#1}{arXiv~#1}}
\providecommand{\doi}[1]{\url{http://dx.doi.org/#1}}
\providecommand{\MR}{\relax\ifhmode\unskip\space\fi MR }
\providecommand{\MRhref}[2]{%
  \href{http://www.ams.org/mathscinet-getitem?mr=#1}{#2}
}
\providecommand{\href}[2]{#2}

\end{document}